\documentclass[preprint,12pt]{elsarticle}
\usepackage{graphicx}
\usepackage{amssymb}
\usepackage{algcompatible}
\usepackage{algpseudocode}
\usepackage{multirow}
\usepackage{multirow}
\usepackage{xspace}
\usepackage{hyperref}
\usepackage{listings}
\usepackage{url}
\usepackage{csquotes}
\usepackage[mathscr]{euscript} 
\usepackage{amsfonts,epsfig}
\usepackage{enumerate}
\usepackage{amsmath,amscd,amsfonts,amssymb}

\usepackage{lineno,hyperref}
\usepackage{graphicx,subfigure}
\usepackage{amsfonts,epsfig}

\usepackage[mathscr]{euscript} 
\usepackage{enumerate}
\usepackage{amsmath,amscd,amsfonts,amssymb}
\usepackage{times}
\usepackage{extarrows}
\usepackage{stmaryrd}
\usepackage{textcomp}
\usepackage[mathscr]{euscript}
\usepackage{algorithm2e}
\usepackage{colortbl}
\usepackage{color,soul}
\usepackage{newunicodechar}
\usepackage{amsthm}
\theoremstyle{definition}
\newtheorem{thm}{Theorem}

\newtheorem{lem}[thm]{Lemma}
\newtheorem{defn}{Definition}%
\newtheorem{exa}{Example}
\newtheorem{rem}[thm]{\bf{Remark}}
\newcommand{\cX}{\underline{\bf X}}
\newcommand{\cQ}{\underline{\bf Q}}
\newcommand{\cC}{\underline{\bf C}}
\newcommand{\cY}{\underline{\bf Y}}
\newcommand{\cI}{\underline{\bf I}}

\newcommand{\cU}{\underline{\bf U}}
\newcommand{\cV}{\underline{\bf V}}

\newcommand{\cR}{\underline{\bf R}}
\newcommand{\cS}{\underline{\bf S}}

\newcommand{\R}{{\bf R}}
\newcommand{\X}{{\bf X}}
\newcommand{\Y}{{\bf Y}}
\newcommand{\I}{{\bf I}}

\newcommand{\C}{{\bf C}}
\newcommand{\Q}{{\bf Q}}
\newcommand{\tP}{{\bf P}}

\newcommand{\W}{{\bf W}}

\newcommand{\U}{{\bf U}}

\newcommand{\V}{{\bf V}}
\newcommand{\cu}{{\bf u}}
\newcommand{\cc}{{\bf c}}
\newcommand{\rr}{{\bf r}}

\usepackage{amssymb}

\journal{Elsevier}

\begin{document}

\begin{frontmatter}



\title{Robust Low-Tubal-rank tensor recovery Using Discrete Empirical Interpolation Method with Optimized Slice/Feature Selection}

 \author[label1]{Salman Ahmadi-Asl}
 \affiliation[label1]{organization={Center for Artificial Intelligence Technology, Skolkovo Institute of Science and Technology, Moscow, Russia,\,s.asl@skoltech.ru}
            }

\author[label1]{Anh-Huy Phan}



\author[label2]{Cesar F. Caiafa}
\affiliation[label2]{organization={Instituto Argentino de Radioastronomia—CCT La Plata, CONICET/CIC-PBA/UNLP, Villa Elisa 1894, Argentina},}

\author[label1]{Andrzej Cichocki}




\begin{abstract}
In this paper, we extend the Discrete Empirical Interpolation Method (DEIM) to the third-order tensor case based on the t-product and use it to select important/significant lateral and horizontal slices/features. The proposed Tubal DEIM (TDEIM) is investigated both theoretically and numerically. The experimental results show that the TDEIM can provide more accurate approximations than the existing methods. An application of the proposed method to the supervised classification task is also presented. 
\end{abstract}




\end{frontmatter}


\section{Introduction}
The singular value decomposition (SVD) \cite{beltrami1873sulle,jordan1874memoire,jordan1875essai,stewart1993early} is costly for the computation of low-rank approximations of large-scale data matrices. To solve this problem, the matrix CUR (MCUR) or cross approximation methods have been proposed where a fast low-rank matrix approximation of a data matrix is computed using some selected  columns and rows  \cite{tyrtyshnikov2000incomplete,goreinov1997theory}. The MCUR approximation is not only useful in terms of time complexity but can also provide interpretable approximations as the factor matrices preserve the properties of the original data matrix, such as nonnegativity, sparsity, or the elements being integers \cite{mahoney2011randomized}. {There are several approaches to sample columns and rows of a  matrix. Generally, one can categorize these sampling methods into {\it randomized} and {\it deterministic} methods. The uniform, length-squared, and leverage-score probability distributions \cite{mahoney2011randomized} as randomized algorithms have been widely used in the literature for sampling columns/rows of a large-scale data matrix. On the other hand, the maxvolume \cite{goreinov2010find} and the discrete empirical interpolation method (DEIM) \cite{sorensen2016deim} are two well-known deterministic sampling algorithms. The first method samples columns/rows in such a way that the volume\footnote{The volume of a square matrix is defined as the absolute value of the determinant of a matrix, whereas the volume of a rectangular matrix is defined as the multiplication of its singular values.} of the intersection matrix is maximized. It was shown that this method can provide almost optimal approximations. The DEIM method 
uses a basis 
e.g., the top left singular vectors, 
to sample rows of a matrix. 
The columns of the matrix $\X^T$ are sampled and treated as the rows of the matrix $\X$. The DEIM is indeed an interpolative MCUR method as the approximation obtained by this method matches the actual columns/rows of the original data matrix in the sampled indices. We should point out that the Cross2D \cite{savostyanov2006polilinear,tyrtyshnikov2000incomplete} is also a deterministic sampling approach\footnote{Except the first stage of the algorithm where the first column is selected randomly, all other stages are performed deterministically.} which sequentially interpolates a given matrix in new sampled columns and rows. However, in contrast to the Cross2D, which for different runs can provide different MCUR approximations, the DEIM method for a given fixed basis matrix always gives the same MCUR approximation. Another benefit of the DEIM method is that for two given data matrices $\X,\Y$, if we can find a shared basis matrix for them, for example by the Generalized SVD (GSVD) \cite{van1976generalizing,paige1981towards}, it can be used to sample columns of the mentioned matrices. Here, the sampled columns of the matrices $\X,\,\Y$ have the same indices because the DEIM method employees the same basis for two datasets. Indeed the authors in \cite{gidisu2022generalized,gidisu2022restricted} use this  trick for generalizing the MCUR to the tensor case.} 
 
The MCUR methods are generalized to tensors in \cite{oseledets2008tucker,oseledets2010tt,caiafa2010generalizing,drineas2007randomized,tarzanagh2018fast}. The methods in \cite{oseledets2008tucker,caiafa2010generalizing,drineas2007randomized} are related to the Tucker model \cite{tucker1964extension,tucker1966some}, the method in  \cite{oseledets2010tt} is for the Tensor Train model \cite{oseledets2011tensor} and the approach proposed in   \cite{tarzanagh2018fast} is for the tensor/tubal SVD (t-SVD) \cite{kilmer2013third}. In this paper, we focus on the tensor SVD (t-SVD), which is defined based on the tubal product (t-product). The t-SVD has similar properties as the classical SVD. In particular, in contrast to the Tucker decomposition \cite{tucker1964extension,tucker1966some} or Canonical polyadic decomposition \cite{hitchcock1927expression,hitchcock1928multiple}, its truncation provides the best low tubal rank approximation in the least-squares sense. The tubal leverage score sampling and uniform sampling are used in \cite{tarzanagh2018fast} and \cite{ahmadi2021cross,ahmadi2022cross}, respectively, for sampling lateral and horizontal slices. To the best of our knowledge, the DEIM method has not been generalized to the tensor case yet based on the t-product, while its better approximation accuracy and robustness with respect to rank variation has been shown for matrices \cite{sorensen2016deim}. Note that the DEIM sampling is used in \cite{saibaba2016hoid} to sample fibers for the computation of low Tucker rank approximation but our work is different from that as it is for the t-SVD model. Also, there is no any paper comparing different sampling approaches for lateral/horizontal slices. This paper aims to investigate more deeply these issues and motivated by the work \cite{sorensen2016deim}, we extend the DEIM method to the t-product, which we refer to it as tubal DEIM (TDEIM). The TDEIM selects the indices of important lateral and horizontal slices. The proposed method outperforms all baseline sampling algorithms including {\it the top leverage scores method}, {\it the leverage score sampling} and {\it the uniform sampling without replacement}.

The key contributions of this work are stated as follows

\begin{itemize}
    \item Extending the discrete empirical
interpolation method (DEIM) to the tensor tubal case based on the t-product and using it to select important/optimized lateral and horizontal slices/features. 

\item Developing a new hybrid TDEIM algorithm  that uses the tubal leverage scores for lateral/horizontal slice sampling.
    
    \item Extensive simulations on synthetic and real-world datasets with an application to the supervised classification task to evaluate the performance and efficiency of the proposed algorithms. 
\end{itemize}

This paper is organized as follows. We first present some basic tensor concepts in Section \ref{Sec:prelim}. The t-SVD is introduced in Section \ref{Sec:tSVD} with some intuitions behind this model. Tha matrix and tensor cross approximation methods are outlined in Section \ref{Sec:MACA}. The Discrete Empirical Interpolation Method (DEIM) is presented in Section \ref{deim} and its extension to tensors is studied in Section \ref{deimt}. The computer simulation results are presented in Section \ref{Sec:sim} and a conclusion is given in Section \ref{Sec:Con}.

\section{Preliminaries}\label{Sec:prelim}
To present the main materials, we first need to introduce the basic notations and definitions. An underlined bold capital letter, a bold capital letter, and a bold lower letter
denote a tensor, a matrix, and a vector, respectively. The subtensors, which are generated by fixing all but two modes are called slices. For a special case of a third-order tensor $\underline{\X},$ we call the slices $\underline{\X}(:,:,k),\,\underline{\X}(:,j,:),\,\underline{\X}(i,:,:)$ frontal, lateral, and horizontal slices. For the brevity of presentation, sometimes for a frontal slice $\cX(:,:,i)$, we use the notation $\cX_i$. Similarly, fibers are generated by fixing 
all but one mode. For a third-order tensor $\cX,$ 
a special type of fiber that is generated by fixing the first and the second modes, e.g. $\underline{\X}(i,j,:)$ is called a tube. {The notation “conj” denotes the component-wise complex conjugate of a matrix.} The notation $\|.\|_F$ stands for the Frobenius norm of tensors/matrices and $|.|$ is used to the denote absolute value of a number. The symbol $\|.\|_2$ denotes the spectral norm of matrices or Euclidean norm of vectors. The Moore–Penrose inverse is denoted by $\dag$. We use the MATLAB notation to denote a subset of a matrix or tensor. For example, for a given data matrix $\X$, by $\X(:,\mathcal{J})$ and $\X(\mathcal{I},:)$ we mean two matrices, which sample a part of rows and columns if the matrix $\X$, respectively where $\mathcal{I}\subset\{1,2,\ldots,I_1\}$ and $\mathcal{J}\subset\{1,2,\ldots,I_2\}$. 

\begin{defn} ({\bf t-product})
Let $\underline{\mathbf X}\in\mathbb{R}^{I_1\times I_2\times I_3}$ and $\underline{\mathbf Y}\in\mathbb{R}^{I_2\times I_4\times I_3}$, the t-product $\underline{\mathbf X}*\underline{\mathbf Y}\in\mathbb{R}^{I_1\times I_4\times I_3}$ is defined as follows
\begin{equation}\label{TPROD}
\underline{\mathbf C} = \underline{\mathbf X} * \underline{\mathbf Y} = {\rm fold}\left( {{\rm circ}\left( \underline{\mathbf X} \right){\rm unfold}\left( \underline{\mathbf Y} \right)} \right),
\end{equation}
where 
\[
{\rm circ} \left(\underline{\mathbf X}\right)
=
\begin{bmatrix}
\underline{\mathbf X}(:,:,1) & \underline{\mathbf X}(:,:,I_3) & \cdots & \underline{\mathbf X}(:,:,2)\\
\underline{\mathbf X}(:,:,2) & \underline{\mathbf X}(:,:,1) & \cdots & \underline{\mathbf X}(:,:,3)\\
 \vdots & \vdots & \ddots &  \vdots \\
 \underline{\mathbf X}(:,:,I_3) & \underline{\mathbf X}(:,:,I_3-1) & \cdots & \underline{\mathbf X}(:,:,1)
\end{bmatrix},
\]
and 
\[
{\rm unfold}(\underline{\mathbf Y})=
\begin{bmatrix}
\underline{\mathbf Y}(:,:,1)\\
\underline{\mathbf Y}(:,:,2)\\
\vdots\\
\underline{\mathbf Y}(:,:,I_3)
\end{bmatrix},\hspace*{.5cm}
\underline{\mathbf Y}={\rm fold} \left({\rm unfold}\left(\underline{\mathbf Y}\right)\right).
\]
As described in \cite{kilmer2011factorization,kilmer2013third}, the t-product is performed via Discrete Fourier Transform (DFT) and in \cite{kernfeld2015tensor} it was suggested to use any invertible transformation rather than the DFT. Later, nonivertible and even nonlinear transformations were used in \cite{jiang2020framelet} and \cite{li2022nonlinear}, respectively. The advantage of using such unitary transformations is the possibility of computing the t-SVD of a data tensor with lower tubal rank \cite{song2020robust,jiang2020framelet} (to be discussed later). {The MATLAB command ${\rm fft}(\cX,[],3)$, computes the DFT of all tubes of the data tensor $\cX$.} The fast version of the t-product is summarized in Algorithm \ref{ALG:TVDP} where the DFT of only the first $\lceil \frac{I_3+1}{2}\rceil$ frontal slices is needed while the original version processes all frontal slices \cite{kilmer2011factorization,kilmer2013third}. 
\end{defn}

\begin{defn} ({\bf Transpose})
The transpose of a tensor $\underline{\mathbf X}\in\mathbb{R}^{I_1\times I_2\times I_3}$ is denoted by $\underline{\mathbf X}^{T}\in\mathbb{R}^{I_2\times I_1\times I_3}$ produced by applying the transpose to all frontal slices of the tensor $\underline{\mathbf X}$ and reversing the order of the transposed frontal slices from the second till the last one.
\end{defn}

\begin{defn} ({\bf Identity tensor})
Identity tensor $\underline{\mathbf I}\in\mathbb{R}^{I_1\times I_1\times I_3}$ is a tensor whose first frontal slice is an identity matrix of size $I_1\times I_1$ and all other frontal slices are zero. It is easy to show $\underline{\I}*\underline{\X}=\underline{\X}$ and $\underline{\X}*\underline{\I} =\underline{\X}$ for all tensors of conforming sizes.
\end{defn}
\begin{defn} ({\bf Orthogonal tensor})
A tensor $\underline{\mathbf X}\in\mathbb{R}^{I_1\times I_1\times I_3}$ is orthogonal if ${\underline{\mathbf X}^T} * \underline{\mathbf X} = \underline{\mathbf X} * {\underline{\mathbf X}^ T} = \underline{\mathbf I}$.
\end{defn}

\begin{defn} ({\bf f-diagonal tensor})
If all frontal slices of a tensor are diagonal then the tensor is called f-diagonal.
\end{defn}

\begin{defn}
({\bf Inverse of a tensor}) The inverse of the tensor $\cX\in\mathbb{R}^{I_1\times I_1\times I_3}$ is denoted by $\cX^{-1}\in\mathbb{R}^{I_1\times I_1\times I_3}$ is a unique tensor satisfying the following equations
\begin{eqnarray*}
\cX^{-1}*\cX=\cX*\cX^{-1}={\bf I},
\end{eqnarray*}
where ${\bf I}$ is an identity tensor of size $I_1\times I_1\times I_3$. The inverse of a tensor can be computed in the Fourier domain very fast as presented in Algorithm \ref{ALG:TSVDP}.
\end{defn}

\RestyleAlgo{ruled}
\LinesNumbered
\begin{algorithm}
\SetKwInOut{Input}{Input}
\SetKwInOut{Output}{Output}\Input{Two data tensors $\underline{\mathbf X} \in {\mathbb{R}^{{I_1} \times {I_2} \times {I_3}}},\,\,\underline{\mathbf Y} \in {\mathbb{R}^{{I_2} \times {I_4} \times {I_3}}}$} 
\Output{t-product $\underline{\mathbf C} = \underline{\mathbf X} * \underline{\mathbf Y}\in\mathbb{R}^{I_1\times I_4\times I_3}$}
\caption{Fast t-product of two tensors \cite{kilmer2011factorization,lu2019tensor}}\label{ALG:TVDP}
      {
      $\widehat{\underline{\mathbf X}} = {\rm fft}\left( {\underline{\mathbf X},[],3} \right)$;\\
      $\widehat{\underline{\mathbf Y}} = {\rm fft}\left( {\underline{\mathbf Y},[],3} \right)$;\\
\For{$i=1,2,\ldots,\lceil \frac{I_3+1}{2}\rceil$}
{                        
$\widehat{\underline{\mathbf C}}\left( {:,:,i} \right) = \widehat{\underline{\mathbf X}}\left( {:,:,i} \right)\,\widehat{\underline{\mathbf Y}}\left( {:,:,i} \right)$;\\
}
\For{$i=\lceil\frac{I_3+1}{2}\rceil+1,\ldots,I_3$}{
$\widehat{\underline{\mathbf C}}\left( {:,:,i} \right)={\rm conj}(\widehat{\underline{\mathbf C}}\left( {:,:,I_3-i+2} \right))$;
}
$\underline{\mathbf C} = {\rm ifft}\left( {\widehat{\underline{\mathbf C}},[],3} \right)$;   
       	}       	
\end{algorithm}

\RestyleAlgo{ruled}
\LinesNumbered
\begin{algorithm}
\SetKwInOut{Input}{Input}
\SetKwInOut{Output}{Output}\Input{The data tensor $\underline{\mathbf X} \in {\mathbb{R}^{{I_1} \times {I_1} \times {I_3}}}$} 
\Output{Tensor Inverse $\underline{\mathbf X}^{-1}\in\mathbb{R}^{I_1\times I_1\times I_3}$}
\caption{Fast Inverse computation of the tensor $\underline{\bf X}$}\label{ALG:TSVDP}
      {
      $\widehat{\underline{\mathbf X}} = {\rm fft}\left( {\underline{\mathbf X},[],3} \right)$;\\
\For{$i=1,2,\ldots,\lceil \frac{I_3+1}{2}\rceil$}
{                        
$\widehat{\underline{\mathbf C}}\left( {:,:,i} \right) = {\rm inv}\,(\widehat{\X}(:,:,i))$;\\
}
\For{$i=\lceil\frac{I_3+1}{2}\rceil+1,\ldots,I_3$}{
$\widehat{\underline{\mathbf C}}\left( {:,:,i} \right)={\rm conj}(\widehat{\underline{\mathbf C}}\left( {:,:,I_3-i+2} \right))$;
}
$\underline{\mathbf X}^{\dag} = {\rm ifft}\left( {\widehat{\underline{\mathbf C}},[],3} \right)$;   
       	}       	
\end{algorithm}

The following identity 
\begin{eqnarray}\label{eq_fou}
\|\underline{\X}\|^2_F=\frac{1}{I_3}\sum_{i=1}^{I_3}\|\widehat{\underline{\X}}(:,:,i)\|_F^2,
\end{eqnarray}
is useful in our error analysis where $\underline{\X}\in\mathbb{R}^{I_1\times I_2\times I_3}$ is a given data tensor and $\widehat{\X}(:,:,i)$ is the $i$-th frontal slice of the tensor $\widehat{\X}={\rm fft}(\X,[],3)$, see \cite{zhang2018randomized}. We will use this identity in our theoretical analyses.

\section{Tensor decompositions based on the t-product and tubal leverage-scores}\label{Sec:tSVD}
The tensor SVD (t-SVD) is a viable tensor decomposition that represents a tensor as the t-product of three tensors. The first and last tensors are orthogonal while the middle tensor is an f-diagonal tensor. The generalization of the t-SVD to tensors of order higher than 3 is done in \cite{martin2013order}. Let $\underline{\X}\in\mathbb{R}^{I_1\times I_2\times I_3}$, then the t-SVD gives the following model 
\[
\underline{\X}\approx \cU*\underline{\bf S}*\cV^T,
\]
where $\cU\in\mathbb{R}^{I_1\times R\times I_3 },\,\underline{\bf S}\in\mathbb{R}^{R\times R\times I_3},$ and $\cV\in\mathbb{R}^{I_2\times R \times I_3}$. The tensors $\cU$ and $\cV$ are orthogonal while the tensor $\underline{\bf S}$ is f-diagonal. {We also refer to $\cU$ and $\cV$, as the $R$ leading left and right singular lateral slices of the tensor $\cX$, respectively.} The procedure of the computation of the t-SVD is presented in Algorithm \ref{ALG:Tsvd}. As can be seen, Algorithm \ref{ALG:Tsvd} only needs the SVD of the first $\lceil \frac{I_3+1}{2}\rceil$ slices in the Fourier domain. This idea was suggested in \cite{hao2013facial,lu2019tensor} taking into account the special structure of discrete Fourier transform, while the original t-SVD algorithm developed in \cite{kilmer2011factorization,kilmer2013third} involves the SVD of all frontal slices. Note that this trick is applicable only for real tensors and for complex tensors, we need to compute the SVD of all frontal slices in the Fourier domain. Naturally, we should utilize this idea to skip redundant computations. 


{
The tubal leverage scores of lateral and horizontal slices can be defined for a third order tensor \cite{tarzanagh2018fast}.}  The tubal leverage scores of horizontal slices to $\cU,$ are defined as follows $l_i=\|\underline{\U}(i,:,:)\|^2_F,\,\,i=1,2,\ldots,I_1.$ The tubal leverage scores of lateral slices can be computed similarly for the tensor $\cV$. Here, we consider $l'_j=\|\underline{\V}(j,:,:)\|^2_F,\,\,j=1,2,\ldots,I_2.$
The important lateral/horizontal slices can be selected according to high tubal leverage scores or using the probability distributions 
\begin{eqnarray}\label{pro}
P_i=\frac{l_i}{R},\,\,i=1,2,\ldots,I_1,\quad P_j=\frac{l'_j}{R},\,\,j=1,2,\ldots,I_2,
\end{eqnarray}
where the fractions $P_i$ and $P_j$ are the probabilities of selecting  the $i$-th horizontal slices and the $j$-th lateral slices, respectively. It is obvious that
\[
\sum_{i=1}^{I_1}\|\cU(i,:,:)\|_F^2=\sum_{j=1}^{I_2}\|\cV(j,:,:)\|_F^2=R,
\]
and the fractions in \eqref{pro} indeed define probability distributions. The authors in \cite{tarzanagh2018fast} used the tubal leverage scores to sample horizontal and lateral slices for low tubal rank approximation. We will use the tubal leverage scores in Section \ref{Sec:sim} as a baseline method to sample horizontal and lateral slices.

\RestyleAlgo{ruled}
\LinesNumbered
\begin{algorithm}
\SetKwInOut{Input}{Input}
\SetKwInOut{Output}{Output}\Input{The data tensor $\underline{\mathbf X} \in {\mathbb{R}^{{I_1} \times {I_2} \times {I_3}}}$ and a target tubal-rank $R$} 
\Output{The truncated t-SVD of the tensor $\cX$ as $\underline{\bf X}\approx\cU*\underline{\bf S}*\cV^T$}
\caption{The truncated t-SVD decomposition}\label{ALG:Tsvd}
      {
      $\widehat{\underline{\mathbf X}} = {\rm fft}\left( {\underline{\mathbf X},[],3} \right)$;\\
\For{$i=1,2,\ldots,\lceil \frac{I_3+1}{2}\rceil$}
{                        
$[\widehat{\underline{\mathbf U}}\left( {:,:,i} \right),\widehat{\underline{\bf  S}}(:,:,i),\widehat{\V}(:,:,i)] = {\rm Truncated\operatorname{-} svd}\,(\widehat{\X}(:,:,i),R)$;\\
}
\For{$i=\lceil\frac{I_3+1}{2}\rceil+1,\ldots,I_3$}{
$\widehat{\underline{\mathbf U}}\left( {:,:,i} \right)={\rm conj}(\widehat{\underline{\mathbf U}}\left( {:,:,I_3-i+2} \right))$;\\
$\widehat{\underline{\mathbf S}}\left( {:,:,i} \right)=\widehat{\underline{\mathbf S}}\left( {:,:,I_3-i+2} \right)$;\\
$\widehat{\underline{\mathbf V}}\left( {:,:,i} \right)={\rm conj}(\widehat{\underline{\mathbf V}}\left( {:,:,I_3-i+2} \right))$;
}
$\underline{\mathbf U}_R= {\rm ifft}\left( {\widehat{\underline{\mathbf U}},[],3} \right)$;
$\underline{\mathbf S}_R= {\rm ifft}\left( {\widehat{\underline{\mathbf S}},[],3} \right)$; 
$\underline{\mathbf V}_R= {\rm ifft}\left( {\widehat{\underline{\mathbf V}},[],3} \right)$
       	}       	
\end{algorithm}


\section{Matrix and tensor CUR approximation methods}\label{Sec:MACA}
Matrix CUR or cross approximation is a popular method for fast low-rank matrix approximation with interpretable factor matrices and linear computational complexity \cite{goreinov1997theory}. It samples individual columns and rows of a data matrix, so it can preserve the properties of the original data matrix such as nonnegativity or sparsity. Let $\X\in\mathbb{R}^{I_1\times I_2}$ be a given data matrix. The CUR approximation seeks the approximation of the form $\X\approx\C\U\R$ where $\C=\X(:,\mathcal{J}),\,\R=\X(\mathcal{I},:)$ and $\mathcal{I}\subset \{1,2,\ldots,I_1\}$ and $\mathcal{J}\subset\{1,2,\ldots,I_2\}$. The optimal middle matrix is 
\begin{eqnarray}\label{midmatrix}
\U=\C^{\dag}\X{\R}^{\dag}.
\end{eqnarray}
This procedure requires one or two passes over the data matrix $\X$ and this depends on how the indices are sampled. For instance, for the case of uniform sampling, we do not need to view the whole data matrix while for the case of the leverage-scores or the DEIM, we need access to the whole data matrix. However, for computing the middle matrix $\U$ in \eqref{midmatrix} both of them require the access to the whole data matrix $\X,$ which is prohibitive for extremely large-scale matrices. To ease the computational complexity, it was proposed to use the Moore-Penrose peseudoinverse of the intersection matrix obtained by crossing the sampled columns and rows, i.e., $\U=(\X(\mathcal{J},{\mathcal{I}}))^{\dag}$ and considering the CUR approximation 
\begin{eqnarray}\label{midlmatrix_2}
\X\approx\C\U\R.
\end{eqnarray}
It is demonstrated in \cite{sorensen2016deim} that the latter approximation indeed interpolates the data matrix $\X$ at the sampled column and row indices. More precisely, if $\W=\X-\C\U\R$, then $\W(\mathcal{I},:)=0$ and $\W(:,\mathcal{J})=0$. This does not necessarily hold if we use the middle matrix $\U$ computed via \eqref{midmatrix}. Moreover, it is proved in\cite{hamm2020perspectives} that for a data matrix $\X$ with exact rank $R$, the CUR approximation in \eqref{midlmatrix_2} is exact provided that ${\rm rank}(\X)={\rm rank}(\X(\mathcal{J},\mathcal{I}))$. 

The columns and rows may be selected either deterministically or randomly for which additive or relative approximation errors can be achieved. In the deterministic case, it is known that the columns or rows with maximum volume can provide almost optimal solutions \cite{goreinov2010find}. The 
 discrete empirical interpolation method (DEIM) is another type of deterministic method to select the columns and rows of a matrix that relies on the top singular vectors\cite{sorensen2016deim}.  The sampling methods based on a prior probability distribution are also widely used in the literature using uniform, length-squared, or leverage-score probability distributions, see \cite{mahoney2011randomized}, for an overview on these sampling approaches. It has been demonstrated that sampling columns with leverage-score probability distribution can provide approximations with relative error accuracy, which is of more interest in practice \cite{drineas2008relative}. The MCUR was extended to the tensor case for different types of tensors decompositions. For example, the authors in \cite{mahoney2011randomized} proposed to sample some columns of the unfolding matrices randomly to approximate the factor matrices of the Tucker decomposition. Also, it is suggested to use the Cross2D method to deterministically sample the fibers instead of random sampling. The DEIM method and leverage score sampling methods are also used in\cite{saibaba2016hoid} to sample columns of the unfolding matrices to approximate the factor matrices. The MCUR is used in \cite{oseledets2010tt} to compute low rank approximation of unfolding matrices to compute the TT approximation. The cross approximation is generalized based on the t-product in \cite{tarzanagh2018fast}, where some horizontal and lateral slices are selected. Here, the slices are selected based on the tubal leverage scores. The uniform sampling without replacements are used in \cite{ahmadi2021cross,ahmadi2022cross} for image/video completion and compression. However, there are only these works on tubal CUR approximation and this problem has not been investigated extensively. In this paper, we extend the DEIM method to the tensor case based on the t-product and extensively compare it with the known sampling algorithms. The results show more accurate approximations of the proposed sampling method compared to the baseline methods as was also reported for the matrix case before.   

\section{Discrete Empirical Interpolation Method (DEIM) for column/row sampling}\label{deim}
The DEIM is a building block in our formulation.
It is a discrete form of its continuous version {\it Empirical Interpolation Method} \cite{barrault2004empirical,chaturantabut2010nonlinear} for model order reduction of nonlinear dynamical systems. The DEIM method captures important variables of the underlying dynamics system. It was then used to sample important columns/rows of matrices to compute a low-rank matrix approximation \cite{sorensen2016deim}. To start with describing the DEIM method, let us introduce the {\it interpolatory projector}, which plays an important role in our analysis.

\begin{defn}
 Assume that $\U\in\mathbb{R}^{I_1\times R}$ is a full-rank matrix and ${\bf p}\in\mathbb{N}^R$ is a set of distinct indices. The interpolatory projector $\mathcal{P}$ is an oblique projector onto the range of ${\bf U}$, which is defined as follows
 \begin{eqnarray}
  \mathcal{P}={\bf U}\left({\bf P}^T{\bf U}\right)^{-1}{\bf P}^T   
 \end{eqnarray}
 where ${\bf P}={\bf I}(:,{\bf p})$ and ${\bf I}$ is the identity matrix of size $I_1\times I_1$.
\end{defn}
Let ${\bf y}=\mathcal{P}{\bf x},$ then as shown in \cite{sorensen2016deim}, an important property of the operator $\mathcal{P}$ is that it preserves the elements of ${\bf x}$ with the indices ${\bf p}$, i.e. 
\begin{eqnarray}
{\bf y}({\bf p})={\bf P}^T{\bf y}={\bf P}^T{\bf U}\left({\bf P}^T{\bf U}\right)^{-1}{\bf P}^T{\bf x}={\bf x}({\bf p}).
\end{eqnarray}
This justifies the name of interpolation as the operator $\mathcal{P}$ can interpolate a vector ${\bf x}$ in the index set ${\bf p}$.

The DEIM algorithm iteratively samples the columns/rows according to the columns of a given matrix basis. This method is summarized in Algorithm \ref{ALG:DEIM}. For a given data matrix $\X\in\mathbb{R}^{I_1\times I_2}$, the DEIM algorithm uses a basis $\U=[{\bf u}_1,{\bf u}_2,\ldots,{\bf u}_R]\in\mathbb{R}^{I_1\times R}$ to sample the indices of important rows of $\X$. It first starts from the first column ${\bf u}_1$ and selects the index of an element with maximum absolute values, that is 
\[
{\bf u}_1(p_1)=\|{\bf u}_1\|_{\infty},
\]
and set ${\bf p}=[p_1]$. Then, a new index, $p_2$, is selected by first computing the residual 
\[
{\bf r}_1={\bf u}_2-\mathcal{P}_1{\bf u_1},
\]
where $\mathcal{P}_1={\bf u}_1\left({\bf P}_1^T{\bf u}_1\right)^{-1}{\bf P}_1^T$ is the interpolatory projector for ${\bf p}$ onto the range of ${\bf u}_1$. Selecting an index of ${\bf r}_1$ with maximum absolute value, i.e. ${\bf r}_1(p_2)=\|{\bf r}_1\|_{\infty},$ and we update the index set as ${\bf p}=[p_1,p_2]$. This procedure is continued by eliminating the
direction of the so-called {\it interpolatory projection} in the former basis vectors from the next one
and again finding the index of the entry with the largest magnitude in the residual vector. To be more precise, assume that we have already sampled $(j-1)$ row indices as ${\bf p}_{j-1}=[p_1,p_2,\ldots,p_{j-1}]$ and we need to select the $j$-th row index. The residual term ${\bf r}_{j-1}$ computed as
\begin{eqnarray}
{\bf r}_j={\bf u}_j-\mathcal{P}_{j-1}{\bf u}_{j}.  
\end{eqnarray}
where 
\begin{eqnarray}
\mathcal{P}_{j-1}&=& {\bf U}_{j-1}\left({\bf P}_{j-1}^T{\bf U}_{j-1}\right)^{-1}{\bf P}_{j-1}^T,\\
{\bf U}_{j-1}&=&[{\bf u}_1,{\bf u}_2,\ldots,{\bf u}_{j-1}],\\
{\bf P}_{j-1}&=&{\bf I}(:,{\bf p}_{j-1}).
\end{eqnarray}
Now, the $j$-th index, which is chosen as ${\bf r}_j(p_j)=\|{\bf r}_j\|_{\infty}.$ It is observed that the DEIM sampling method requires that the matrix ${\bf P}_{j-1}\U_{j-1}$ to be nonsingular at each iteration. {This is demonstrated in \cite{sorensen2016deim} under the condition that U is of full rank, which was our supposition}
Note that the DEIM method is basis dependent but for two different bases, $\Q$ and $\U$, that ${\rm Range}(\U)={\rm Range}(\Q)$, the DEIM approach provides the same indices \cite{chaturantabut2010nonlinear} 
\[
{\bf U}\left({\bf P}^T{\bf U}\right)^{-1}{\bf P}^T={\bf Q}\left({\bf P}^T{\bf Q}\right)^{-1}{\bf P}^T,
\]
\begin{rem} \cite{sorensen2016deim}
At the iteration $j$, we have ${\bf r}_{j-1}({\bf p}_{j-1})=0,$ because ${\bf P}_{j-1}{\bf u}_j$ matches ${\bf u}_j$ in the indices ${\bf p}_{j-1}$. This guarantees that each iteration samples distinct indices.  
\end{rem}
The error bound of the approximation obtained by the DEIM is presented in the next lemma.
\begin{lem}\label{Soren_lem}\cite{sorensen2016deim,saibaba2016hoid}
Assume ${\bf P}^T\U$ is invertible and let $\mathcal{P}$ be the interpolatory projector $\mathcal{P}=\U({\bf P}^T\U)^{-1}{\mathbf{P}}^T$. If
$\U^T\U = \I$, then any $\X\in\mathbb{R}^{I_1\times I_2}$ satisfies
\begin{eqnarray}
 \|\X-\mathcal{P}\X\|^2_F\leq \|({\bf P}^T\U)^{-1}\|^2_2\|(\I-\U\U^T)\X\|^2_F, 
\end{eqnarray}
Additionally, if $\U$ consists of the $R$ leading left singular vectors of $\X$, then
\begin{eqnarray}
    \|\X-\mathcal{P}\X\|^2_F\leq\|({\bf P}^T\U)^{-1}\|^2_2\|(\I-\U\U^T)\X\|^2_F\leq\|({\bf P}^T\U)^{-1}\|^2_2\sum_{t>R}\sigma^2_{t}.
\end{eqnarray}
\end{lem}
\
The same result can be stated for the column selection process as follows
\begin{eqnarray}
 \|\X-\X\mathcal{W}\|^2_F\leq \|(\V^T\Q)^{-1}\|^2_2\|\X(\I-\V\V^T)\|^2_F\leq\|(\V^T\Q)^{-1}\|^2_2\sum_{t>R}\sigma^2_{t},
\end{eqnarray}
where $\mathcal{W}=\V(\Q^T\V)^{-1}\Q^T$ and $\Q=[{\bf e}_{q_1},{\bf e}_{q_2},\ldots,{\bf e}_{q_j}]$ is a collection of standard unit vectors corresponding to the row index set ${\bf q}=[q_1,q_2,\ldots,q_j]$ and ${\bf V}$ is a basis for the row space of the matrix $\X$. We see that the following quantities 
\begin{eqnarray}\label{errconst}
 \eta_p=\|(\tP^T\U)^{-1}\|^2_2, \quad \eta_q=\|(\V^T\Q)^{-1}\|^2_2,   
\end{eqnarray}
play important roles in the upper error bounds. So, the
conditioning of the problem heavily depends on these quantities and we are interested in sampling algorithms with these quantities being as small as possible. For the upper bounds of the mentioned quantities, see \cite{sorensen2016deim}. The next lemma demonstrates the upper bound on a CUR approximation obtained by DEIM for the row/column selection.  
\begin{lem}\label{LemSoren}\cite{sorensen2016deim}
Suppose $\X\in\mathbb{R}^{I_1\times I_2}$ that for the row and column indices ${\bf p}$ and ${\bf q}$, the matrices $\C=\X(:,{\bf q})=\X\Q$
and $\R=\X({\bf p},:)=\tP\X$, are full-rank where
\[
{\bf P}=[{\bf e}_{p_1},{\bf e}_{p_2},\ldots,{\bf e}_{p_j}],\quad\Q=[{\bf e}_{q_1},{\bf e}_{q_2},\ldots,{\bf e}_{q_j}],
\]
with finite error constants $\eta_p$ and $\eta_q$, defined in \eqref{errconst} and set $\U={\C}^{\dag}\X\R^{\dag}$, where $1\leq R<\min(I_1,I_2)$. Then
\begin{eqnarray}
    \|\X-\C\U\R\|^2_2\leq (\eta_p+\eta_q)\sigma^2_{R+1}.
\end{eqnarray} 
\end{lem}
Using the sub-multiplicative property of the Frobenius norm, we have
\begin{eqnarray}
    \|\X-\C\U\R\|^2_F\leq (\eta_p+\eta_q)\sum_{t>R}\sigma^2_{t}.
\end{eqnarray}  
\section{Tubal Discrete Empirical Interpolation Method (TDEIM) for lateral/horizontal slice sampling}\label{deimt}
It is empirically shown in \cite{sorensen2016deim}, that the DEIM algorithm outperforms the leverage-score sampling method 
as one of the best sampling approaches.
This motivates us to generalize it to the tensor case based on the t-product. 
In this section, we discuss how to perform this generalization properly. A link between the tubal DEIM and the tubal leverage score sampling is also studied. 
We we call the extended methods as the tubal DEIM (TDEIM). Similar to the DEIM, a key concept of the TDEIM is the {\it interpolatory projector} that we now define it. For a given set of $R$ indices ${\bf s}\in\mathbb{N}^R$, a full tubal-rank tensor\footnote{A tensor with linear independent lateral slices, for example, the tensor $\cU$ obtained from the t-SVD can be used.} $\cU\in\mathbb{R}^{I_1\times R\times I_3},$ consider $\underline{\bf S}=\underline{\bf I}(:,{\bf s},:)\in\mathbb{R}^{I_1\times R\times I_3}$, as the selection tensor where $\underline{\bf I}\in\mathbb{R}^{I_1\times I_1\times I_3}$ is an identity tensor. The tensorial oblique projection operator is defined as follows
\begin{eqnarray}\label{proj}
\underline{\mathcal{P}}=\underline{\bf U}*(\underline{\bf{S}}^T*\underline{\bf{U}})^{-1}*\underline{\bf{S}}^T. 
\end{eqnarray}

\begin{defn}
{
({\bf Horizontal slice sampling}) Let $\cX\in\mathbb{R}^{I_1\times I_2\times I_3}$ be a given tensor and ${\bf s}\in\mathbb{N}^R$ be an index set. The subtensor $\cX({\bf s},:,:)\in\mathbb{R}^{R\times I_2\times I_3}$ that collects some horizontal slices of the tensor $\cX$ is refered to as the horizontal slice sampling tensor.
Assume $\underline{\bf S}=\underline{\bf I}(:,{\bf s},:)\in\mathbb{R}^{I_1\times R\times I_3}$, then the horizontal slice sampling is equivalent to $\cX({\bf s},:,:)=\underline{\bf S}^T*\cX.$ }
\end{defn}
Given an arbitrary tensor $\underline{\bf G}\in\mathbb{R}^{I_1\times I_2\times I_3}$ and $\cX=\underline{\mathcal{P}}*\underline{\bf G}$, we have  
\begin{eqnarray}
\nonumber
\cX({\bf s},:,:)&=&\underline{\bf S}^T*\cX=
\underline{\bf S}^T*\underline{\bf U}*(\underline{\bf{S}}^T*\underline{\bf{U}})^{-1}*\underline{\bf{S}}^T*\underline{\bf G}\\
&=&
\underline{\bf{S}}^T*\underline{\bf G}=\underline{\bf G}({\bf s},:,:).
\end{eqnarray}
This means that the projection operator $\underline{\mathcal{P}}$ preserves the horizontal slices of $\underline{\bf G}$ specified by the index set ${\bf s}$.

{Let us describe the TDEIM for sampling horizontal slices of a given tensor. The TDEIM begins by the first lateral slice of the basis tensor, i.e. $\cU(:, 1,:)$, and select the index with the highest Euclidean norm among all its tubes, e.g. $
s_1=\arg\max_{1\leq i \leq I_1}\|\underline{\bf U}(i,1,:)\|$. The index of the first sampled horizontal slice will be $s_1$.
}
The subsequent indices are selected according to the indices with the maximum Euclidean norm of the tubes of the residual lateral slice that is computed by removing the direction of the interpolatory projection in the previous basis horizontal slice from the subsequent one. To be more specific, let the selected indices be ${\bf s}_{j-1}=\{s_1,s_2,\ldots,s_{j-1}\}$ and we want to select the new index $s_j$. To do so, we compute the residual slice
\[
\underline{\bf R}(:,j,:)=\underline{\bf U}(:,j,:)-\underline{\mathcal{P}}_{j-1}*\underline{\bf U}(:,j,:),
\]
where $\underline{\mathcal{P}}_{j-1}=\underline{\bf U}^{j-1}*({\underline{\bf S}^{j-1}}^T*\underline{{\bf U}}^{j-1})^{-1}*{\underline{\bf S}^{j-1}}^T,\,\underline{\bf S}^{j-1}=\underline{\bf I}(:,{\bf s}_{j-1},:),$ and $\underline{\bf U}^{j-1}=\underline{\bf U}(:,{\bf s}_{j-1},:)$. Then, a new sampled horizontal slice index $s_j$ with the maximum Euclidean norm of tubes is computed or
\[
s_j=\arg\max_{1\leq i \leq I_1}\|\underline{\bf R}(i,j,:)\|_2.
\]
The same process can be carried out to select lateral slices where the basis tensor $\cV$ should be used. The TDEIM method for sampling horizontal slices is summarized in Algorithm \ref{ALG:DEIM_ten}. With a slight modification of Algorithm \ref{ALG:DEIM_ten}, it can be used for lateral slice sampling. Similar to Lemma \ref{Soren_lem}, the next theorem demonstrates the error bound of the approximation yielded by the TDEIM method. 

\begin{lem}\label{lem_tub}
Assume $\underline{\mathcal{P}}=\cU*(\cS^T*\cU)^{-1}*\cS^T$ and ${\underline{\bf S}}^T*\cU$ is invertible. If $\cU^T*\cU=\cI,$ then any tensor $\cX\in\mathbb{R}^{I_1\times I_2\times I_3}$ satisfies 
\begin{eqnarray}
\|\cX-\underline{\mathcal{P}}*\cX\|^2_F\leq\frac{1}{I_3}\max_{i}\,\left({\|(\widehat{\cS}^T_i\widehat{\cU}_i)^{-1}\|_2^2}\right)\sum_{i=1}^{I_3}\|(\I-\widehat{\cU}_i\widehat{\cU}_i^T)\widehat{\X}_i\|_F^2,  
\end{eqnarray}
where $\widehat{\cX}={\rm fft}(\cX,[],3),\,\widehat{\cU}={\rm fft}(\cU,[],3)$ and $\widehat{\cS}={\rm fft}(\cS,[],3)$. Here, $\cX_i=\cX(:,:,i),\,\underline{\mathbf S}_i=\underline{\mathbf S}(:,:,i)$ and $\cU_i=\cU(:,:,i)$. Moreover, if $\cU$ contains the $R$ leading left singular lateral slices of the tensor $\cX$, then  
\begin{eqnarray}
\|\cX-\underline{\mathcal{P}}*\cX\|^2_F\leq\frac{1}{I_3}\max_{i}\,\left({\|(\widehat{\cS}^T_i\widehat{\cU}_i)^{-1}\|_2^2}\right)\sum_{i=1}^{I_3}\sum_{t>R}^{}{(\sigma^i_{t})^2},  
\end{eqnarray}
where $\sigma^i_{t}$ are the $t$-th largest singular values of the frontal slice $\widehat{\cX}_i=\widehat{\cX}(:,:,i)$.
\end{lem}

\begin{proof}
From identity \eqref{eq_fou}, we have 
\begin{eqnarray}
\|\cX-\underline{\mathcal P}*\cX\|_F^2=\frac{1}{I_3}\sum_{i=1}^{I_3}\|\widehat{\cX}_i-\widehat{\underline{\mathcal P}}_i\widehat{\cX}_i\|_F^2,    
\end{eqnarray}
where $\widehat{\underline{\mathcal P}}={\rm fft}({{\mathcal P}},[],3)$ and $\widehat{\underline{\mathcal P}}_i=\widehat{\underline{\mathcal P}}(:,:,i)$. Using Lemma \ref{Soren_lem}, we arrive at
\begin{eqnarray}
 \nonumber
 \|\cX-{\underline{\mathcal P}}*\cX\|_F^2&=&\frac{1}{I_3}\sum_{i=1}^{I_3}\|\widehat{\cX}_i-\widehat{\underline{\mathcal P}}_i\widehat{\cX}_i\|_F^2\\
 \nonumber
 &\leq&
 \frac{1}{I_3}\sum_{i=1}^{I_3}\|(\widehat{\cS}_i^T\widehat{\cU}_i)^{-1}\|_2^2\|(\I-\widehat{\cU}_i\,\widehat{\cU}_i^T)\widehat{\cX}_i\|_F^2\\
 \nonumber
 &\leq&
 \frac{1}{I_3}\max\,\left({\|(\widehat{\cS}^T_i\widehat{\cU}_i)^{-1}\|_2^2}\right)\sum_{i=1}^{I_3}\|(\I-\widehat{\cU}_i\,\widehat{\cU}_i^T)\widehat{\cX}_i\|_F^2,
\end{eqnarray}
taking into account that $\widehat{\underline{\mathcal P}}_i=\widehat{\cU}_i(\widehat{\cS}^T_i\widehat{\cU}_i)^{-1}\widehat{\cS}^T_i$ where $\widehat{\cS}={\rm fft}(\cS,[],3)$ and $\widehat{\cS}_i=\widehat{\cS}(:,:,i)$. So, the proof of the first part is completed. It suffices to consider $\|(\cI-\widehat{\cU}_i\widehat{\cU}_i^T)\widehat{\X}_i\|_F^2=\sum_{t>R}^{}(\sigma_{t}^i)^2$ and substitute it into the last equation, to get the second part of the theorem.
\end{proof}

Similar to the matrix case, let us introduce two quantities as follows
\begin{eqnarray}\label{t_const}
 \tilde{\eta}_p=\frac{1}{I_3}\max_{i}\,\left({\|(\widehat{\cS}^T_i\widehat{\cU}_i)^{-1}\|_2^2}\right),\quad\tilde{\eta}_q=\frac{1}{I_3}\max_{i}\,\left({\|(\widehat{\cV}^T_i\widehat{\cQ}_i)^{-1}\|_2^2}\right), 
\end{eqnarray}
 that will be used in the next theorem. Here, $\cV\in\mathbb{R}^{I_2\times R\times I_3}$ is a tensor basis for the subspace of lateral slices of the original data tensor, $\underline{\bf Q}=\underline{\bf I}(:,{\bf s}_{j-1},:)$ is a tensor of some sampled lateral slices of the identity tensor specified by the index set ${\bf s}_{j-1}=\{s_1,s_2,\ldots,s_{j-1}\},\widehat{\cQ}={\rm fft}(\cQ,[],3),\,\widehat{\cV}={\rm fft}(\cV,[],3)$ and $\widehat{\cQ}_i=\widehat{\cQ}(:,:,i),\,\widehat{\cV}_i=\widehat{\cV}(:,:,i)$.   

\begin{thm}\label{thm_tub}
Suppose $\cX\in\mathbb{R}^{I_1\times I_2\times I_3}$ and $1\leq R<\min(I_1,I_2)$. Assume that  horizontal slice and lateral slice indices ${\bf p}$ and ${\bf q}$ give full tubal rank tensors $\cC=\cX(:,{\bf q},:)=\cX*\underline{\mathcal Q}$
and $\cR=\cX({\bf p},:,:)={\underline{\mathcal P}}*\cX$, where $\underline{\mathcal{P}}$ and $\underline{\mathcal{Q}}$ are tensorial interpolation projectors for horizontal and lateral slice sampling, respectively with corresponding finite error constants $\tilde{\eta}_p,\,\tilde{\eta}_q$ defined in \eqref{t_const} and set $\cU=\underline{\C}^{\dag}*\cX*\cR^{\dag}$. Then
\begin{eqnarray}
    \|\cX-\cC*\cU*\cR\|^2_F\leq (\tilde{\eta}_p+\tilde{\eta}_q)\sum_{i=1}^{I_3}\sum_{t>R}(\sigma^i_{t})^2,
\end{eqnarray} 
where $\sigma^i_{t}$ are the $t$-th largest singular values of the frontal slice $\widehat{\cX}_i=\widehat{\cX}(:,:,i)$.
\end{thm}

\begin{proof}
Using identity \eqref{eq_fou}, we have
\begin{eqnarray}
\nonumber
\|\cX-\cC*\cU*\cR\|_F^2&=&\frac{1}{I_3}\sum_{i=1}^{I_3}\|\widehat{\cX}-\widehat{\cC}_i\,\widehat{\cU}_i\,\widehat{\cR}_i\|_F^2,
\end{eqnarray}
where $\widehat{\cC}={\rm fft}(\cC,[],3),\,\widehat{\cU}={\rm fft}(\cU,[],3),\,\widehat{\cR}={\rm fft}(\cR,[],3)$ and $\widehat{\cC}_i=\widehat{\cC}(:,:,i),\,\widehat{\cU}_i=\widehat{\cU}(:,:,i),\,\widehat{\cR}_i=\widehat{\cR}(:,:,i).$ Then, the result can be readily concluded from Lemma \ref{LemSoren}.   
\end{proof}

Theorem \ref{thm_tub} shows that the TDEIM approximation with the middle tensor defined above can provide approximation within a factor $\tilde{\eta}_p+\tilde{\eta}_q$ of the best tubal rank $R$ approximation. It also indicates that the conditioning of the problem depends on these two quantities and the lateral/horizontal slices should be selected in such a way that these quantities be controlled. In the simulation section \eqref{Sec:sim}, we will show that the proposed TDEIM provide lower values for the quantities $\tilde{\eta}_p$ and $\tilde{\eta}_q$ and they change smoothly. 

\RestyleAlgo{ruled}
\LinesNumbered
\begin{algorithm}
\SetKwInOut{Input}{Input}
\SetKwInOut{Output}{Output}\Input{ $\U\in\mathbb{R}^{I_1\times R}$ with $R\leq I_1$ (linearly independent columns)} 
\Output{Indices ${\bf s}\in\mathbb{N}^R$ with distinct entries in $\{1,2,\ldots,I_1\}$}
\caption{DEIM index selection for row selection \cite{sorensen2016deim}}\label{ALG:DEIM}
      {
      $\cu=\U(:,1)$;\\
$s_1=\arg\max_{1\leq i\leq I_1}|\cu_i|$\\
\For{$j=2,3,\ldots,R$}{
$\cu=\U(:,j)$;\\
$\cc=\U(s,1:j-1)^{-1}\cu(s)$;\\
$\rr=\cu-\U(:,1:j-1)\cc$;\\
$s_j=\arg\max_{1\leq i\leq I_1}|\rr_i|$
} 
       	}       	
\end{algorithm}

\RestyleAlgo{ruled}
\LinesNumbered
\begin{algorithm}
\SetKwInOut{Input}{Input}
\SetKwInOut{Output}{Output}\Input{ $\cU\in\mathbb{R}^{I_1\times R\times I_3}$ with $R\leq I_1$ (linearly independent lateral slices)} 
\Output{Indices ${\bf s}\in\mathbb{N}^R$ with distinct entries in $\{1,2,\ldots,I_1\}$}
\caption{Proposed Tubal DEIM (TDEIM) index selection approach for horizontal slice selection}\label{ALG:DEIM_ten}
      {
$s_1=\arg\max_{1\leq i\leq I_1}\|\cU(i,1,:)\|_2$\\
\For{$j=2,3,\ldots,R$}{
$\underline{\C}=\cU({\bf s},1:j-1,:)^{-1}*\cU({\bf s},j,:)$;\\
$\underline\R=\cU(:,j,:)-\cU(:,1:j-1,:)*\underline{\C}$;\\
$s_j=\arg\max_{1\leq i\leq I_1}\|\underline{\R}(i,j,:)\|_2$
} 
       	}       	
\end{algorithm}

\subsection{Faster TDEIM Algorithm}\label{Subsec:HTDEM}
Although the TDEIM can be used to choose the indices of lateral and horizontal slices, its primary restriction is that the maximum number of indices that can be chosen must match the specified tubal rank $R$ of the tensor $\cX$. Using a larger $R$, we need to compute larger tensor basis $\underline{\bf U}\in\mathbb{R}^{I_1\times R\times I_3}$ and this requires higher computational complexity and memory resources, which makes the algorithm prohibitive for big data tensors. It is suggested in \cite{gidisu2022hybrid}, to combine the DEIM with the leverage scores approach to find more than $R$ column/row indices. Let $\X\in\mathbb{R}^{I_1\times I_2}$ be given. The idea is to start with a small rank $R$ where $R<R'<\min(m,n)$ and apply the DEIM to select $R$ rows of the matrix $\X$. Then, more $R'-R$ indices are selected according to the leverage scores of the residuals. Since $R<R'$, we only compute a smaller set of left singular vectors $R$ rather than $R'$, so, we can speed up the computations. We can straightforwardly generalize this idea to tensors and this method is summarized in Algorithm \ref{ALG:DEIM_tenlever}. Indeed, we need to only compute the left singular tensor $\underline{\bf U}$ of size $I_1\times R\times I_3$ rather than $I_1\times R'\times I_3$, and the rest of indices are sampled using the tubal leverage scorers of the residual tensor $\underline{\bf U}$ in Lines 5-7 of Algorithm \ref{ALG:DEIM_tenlever}.
\begin{rem}
The basis tensors $\cU$ and $\cV$ are required in Algorithms \ref{ALG:DEIM} and \ref{ALG:DEIM_ten} can be computed very fast through the randomized truncated t-SVD \cite{zhang2018randomized,ahmadi2022efficient,ahmadi2022randomized}. This version can be regarded as a randomized version of the TDEIM algorithm.
\end{rem}

\RestyleAlgo{ruled}
\LinesNumbered
\begin{algorithm}
\SetKwInOut{Input}{Input}
\SetKwInOut{Output}{Output}\Input{ $\cU\in\mathbb{R}^{I_1\times R\times I_3}$ and $\cV\in\mathbb{R}^{I_2 \times R \times I_3}$ with a target tubal rank $R$ with $R\leq R'\leq \min\{I_1,I_2\}$} 
\Output{Index set ${\bf s}\in\mathbb{N}^R$ and ${\bf p}\in\mathbb{N}^R$with distinct entries in $\{1,2,\ldots,I_1\}$ and $\{1,2,\ldots,I_2\}$ for the selected horizontal and lateral slices}
\caption{Proposed hybrid tubal Leverage Scores and TDEIM (HTDEIM)}\label{ALG:DEIM_tenlever}
      {
\For{$j=1,2,\ldots,R$}{
${\bf s}(j)=\max_{1\leq i\leq I_1}\|\cU(i,j,:)\|$;\\
$\cU=\cU-\cU({\bf s},1:j-1,:)^{-1}*\cU({\bf s},j,:)$;\\
}
Compute the tubal leverage scores $l_i=\|\cU(i,:,:)\|^2,\,\,\,i=1,2,\ldots,I_1,$ and sort $l$ in non-increasing order;\\
Delete components in $l$ corresponding to the indices in ${\bf s}$;\\
Sample $s'=R'-R$ indices corresponding to $R'-R$ largest entries of $l$;\\
${\bf s}=[{\bf s};{\bf s}']$;\\
Perform 1–8 on $\cV$ to get index set ${\bf p}$
       	}       	
\end{algorithm}

\section{Experiments}\label{Sec:sim}
This section presents numerical results to demonstrate the performance of our proposed algorithms. We have implemented and ran the proposed algorithm in MATLAB on a computer with 2.60 GHz Intel(R) Core(TM) i7-5600U processor and 8GB memory. 
We consider four experiments. We use synthetic and optimization based tensors in the first simulation. The second and third examples are for image and video approximations, respectively. In the last experiment, we use the proposed tubal sampling approach to compute a light-weight model for the supervised classification task on the MNIST dataset.
We mainly compare the results achieved by the proposed TDEIM with three sampling algorithms as follows:
\begin{itemize}
    \item Top tubal leverage scores \cite{tarzanagh2018fast}
    \item Tubal leverage score sampling \cite{tarzanagh2018fast}
    \item Uniform sampling without replacement \cite{ahmadi2021cross}
\end{itemize}
The \textit{top leverage score} method first computes the tubal leverage scores of horizontal and lateral slices as described in Section \ref{Sec:tSVD} and then selects the indices corresponding to the $R$ top horizontal and lateral leverage scores. The \textit{Tubal leverage score sampling} method builds the probability distributions \eqref{pro} and samples the horizontal and lateral slices based on them. The \textit{Uniform sampling without replacement} applies the uniform sampling for selecting the horizontal and lateral slices. We perform this procedure without replacement as it is not required to select a horizontal or a lateral slice multiple times. It has been experimentally reported that uniform sampling without replacement works better than the one with replacement. We use the `` Absolute Frobenius Error'' metric defined as follow
\[
{\rm Error}=\|\cX-\cC*\cU*\cR\|_F,
\]
to compare the performance of the proposed sampling approach with the three considered baseline methods.
\begin{exa}\label{exa_1} ({\bf Synthetic and optimization based tensors}) 
In this experiment, we use synthetic and optimization based tensors in our simulations. First consider following synthetic data tensor of size $300\times 400\times 300$
\begin{eqnarray}\label{rand_ten}
\cX(i,j,k)=\frac{1}{(i^p+j^p+k^p)^{1/p}},\\
\nonumber
1\leq i,k\leq 300,\,1\leq j\leq 400.
\end{eqnarray}
We mainly use $p=3$ and $p=5$ in the simulations and top singular tensors tubal of tubal rank $R=15$ for building the tubal leverage scores and also as a basis for using TDEIM Algorithm. The tubal leverage scores corresponding to the data tensor \eqref{rand_ten} for $p=5$ and $p=3$ ar displayed in Figures \ref{Fig_1} and \ref{Fig_3}, respectively. Here, we also show the indices, which were selected using the TDEIM for horizontal and lateral slice selection. Th error achieved by the proposed and the three baselines are reported in Figure \ref{Fig_2}. As can be seen, the proposed algorithm provides lower errors and also is monotonically decreasing while the tubal rank is increasing.This experiment clearly shows that the proposed algorithm is robust and can achieve better results than the baseline sampling algorithms. We see that the TDEIM samples most of the indices with high tubal leverage scores but not all of them contrary to the top tubal leverage scores. This indicates that some indices with lower tubal leverage scores are also crucial for getting more accurate results. This experiment also shows that the uniform sampling or top tubal leverage scores could be unstable while the tubal leverage score sampling provided better results. Similar numerical results were reported in \cite{sorensen2016deim}. The error constants $\tilde{\eta}_p$ and $\tilde{\eta}_q$ for different tubal ranks are displayed in Figure \ref{eta}. This figure demonstrates the mentioned error constants are increasing smoothly, which leads to lower error bound and better conditioning.      

In the second set of experiments, we consider the {\it Exponential, Rastrigin, Booth, Matyas and Easom functions} as two-dimensional functions defined as follows
\begin{eqnarray}
 f({\bf x},{\bf y})&=&-\exp(-0.5({\bf x}^2+{\bf y}^2)),\\
  f({\bf x},{\bf y})&=&20+{\bf x}^2+{\bf y}^2-10(\cos(2\pi{\bf x})+\cos(2\pi{\bf y}))),\\
   f({\bf x},{\bf y})&=&({\bf x}+2{\bf y}-7)^2+(2{\bf x}+{\bf y}-5),\\
   f({\bf x},{\bf y})&=&0.26({\bf x}^2+{\bf y}^2)-0.48{\bf x}{\bf y},\\
   f({\bf x},{\bf y})&=&-\cos({\bf x})\cos({\bf y})\exp(-(({\bf x}-\pi)^2+({\bf y}-\pi)^2)).
\end{eqnarray}
respectively, which are widely used as baseline in optimization. We discritize the mentioned functions over the $[0,1000]\times[0,1000]$ to build matrices of size $1000\times 1000$ and then reshape them to third order tensors of size $100\times 100\times 100$. Here, we apply the proposed TDEIM algorithm and the baseline methods to the generated data tensors with the tubal rank $R=10$. The error achieved by the algorithms are shown in Table \ref{table_func}. In most of the cases the proposed TDEIM method provided the best results.

Let us now compare the running time of the HTDEIM (Algorithm \ref{ALG:DEIM_tenlever}) with the TDEIM (Algorithm \ref{ALG:DEIM_ten}). As we discussed in Section \eqref{Subsec:HTDEM}, the TDEIM requires the basis tensors $\cU$ and $\cV$ for a given tubal rank $R$ to select $R$ lateral and horizontal slices. For a large $R$, computing these data tensors could be expensive and to circumvent this issue, we can use smaller tubal rank $R'<R$ and use the TDEIM to sample $R'$ horizontal and lateral slices. The rest of slices $(R-R')$ are selected according to the top tubal leverage scores of the residual tensor (see subsection \ref{Subsec:HTDEM}). The running times of the HTDEIM and TDEIM algorithms for the data tensor \eqref{rand_ten} are compared in Table \ref{Table1}. It is seen that the HTDEIM algorithm requires less execution times compared to the TDEIM method. The errors achieved by the HTDEIM and TDEIM are compared in Table \ref{Table2}. So, the HTDEIM algorithm delivers quicker results that are comparable to those of the TDEIM approach.



\begin{figure}
\begin{center}
\includegraphics[width=0.5\columnwidth]{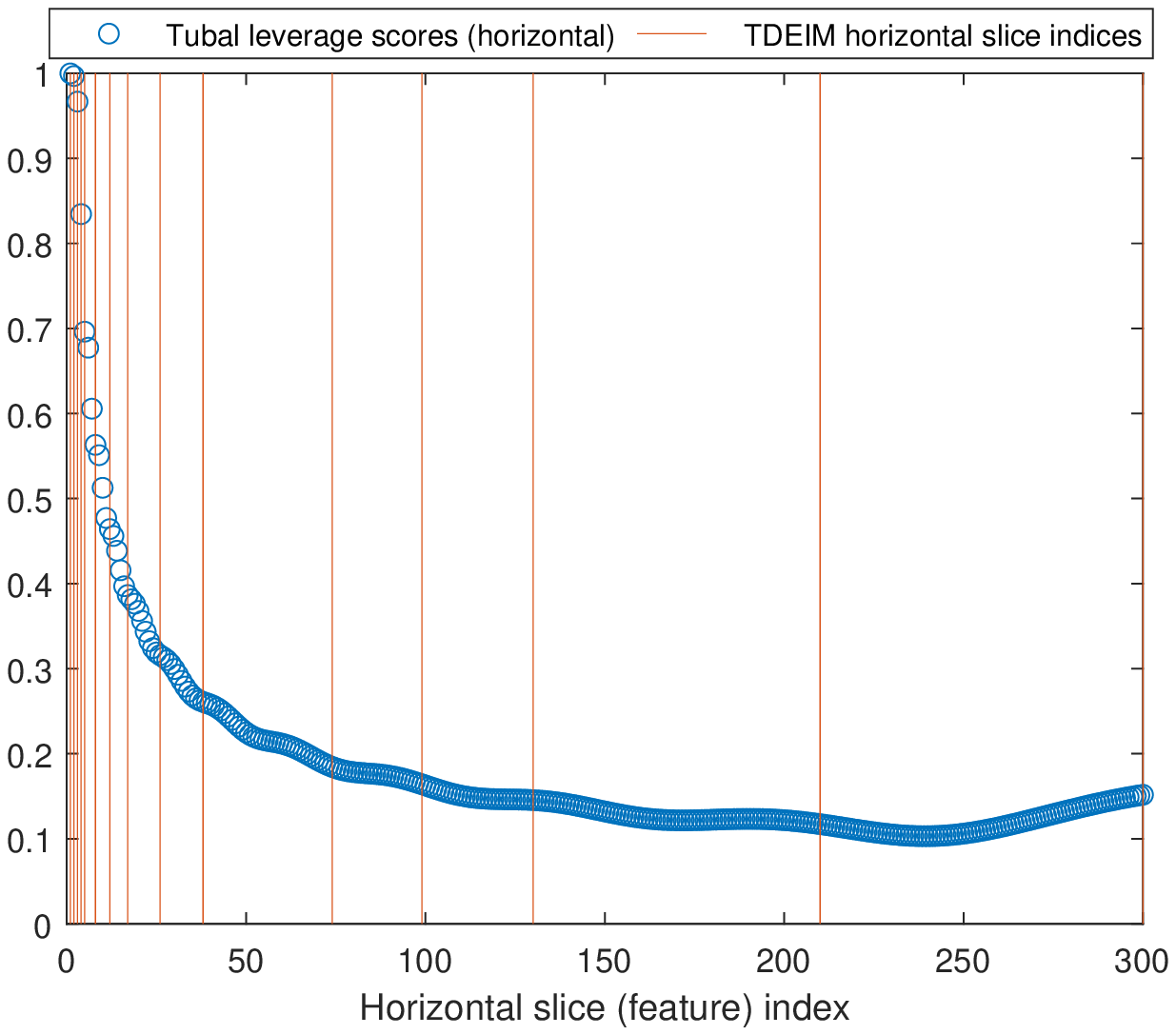}\includegraphics[width=0.5\columnwidth]{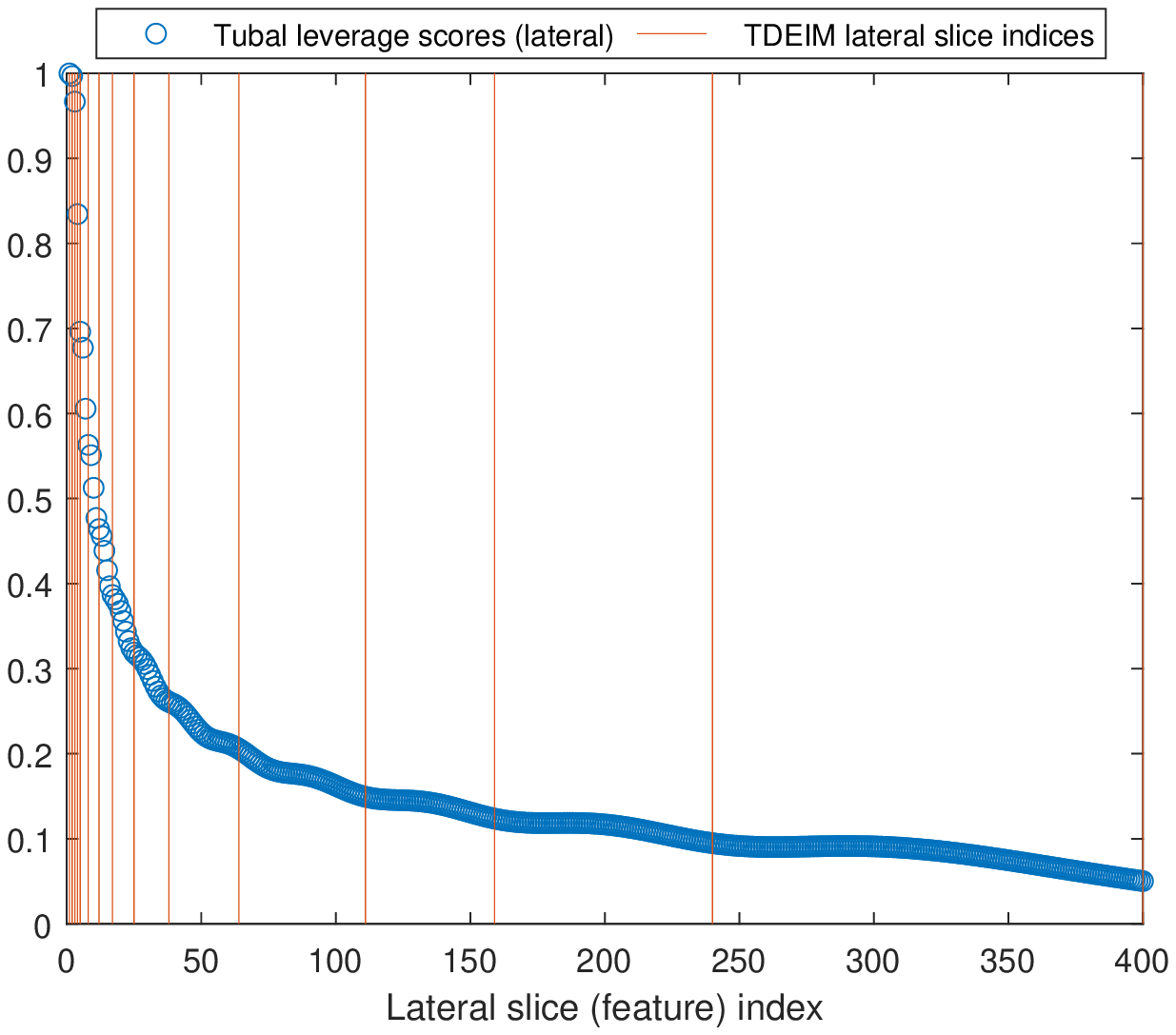}\\
\caption{\small{ ({\bf Left}) The selected horizontal slice indices ({\bf Right}) The selected lateral slice indices 
 for $p=5$. Top singular tensors of tubal rank $R=15$ were used for Example \ref{exa_1}.}}\label{Fig_1}
\end{center}
\end{figure}

\begin{figure}
\begin{center}
\includegraphics[width=0.5\columnwidth]{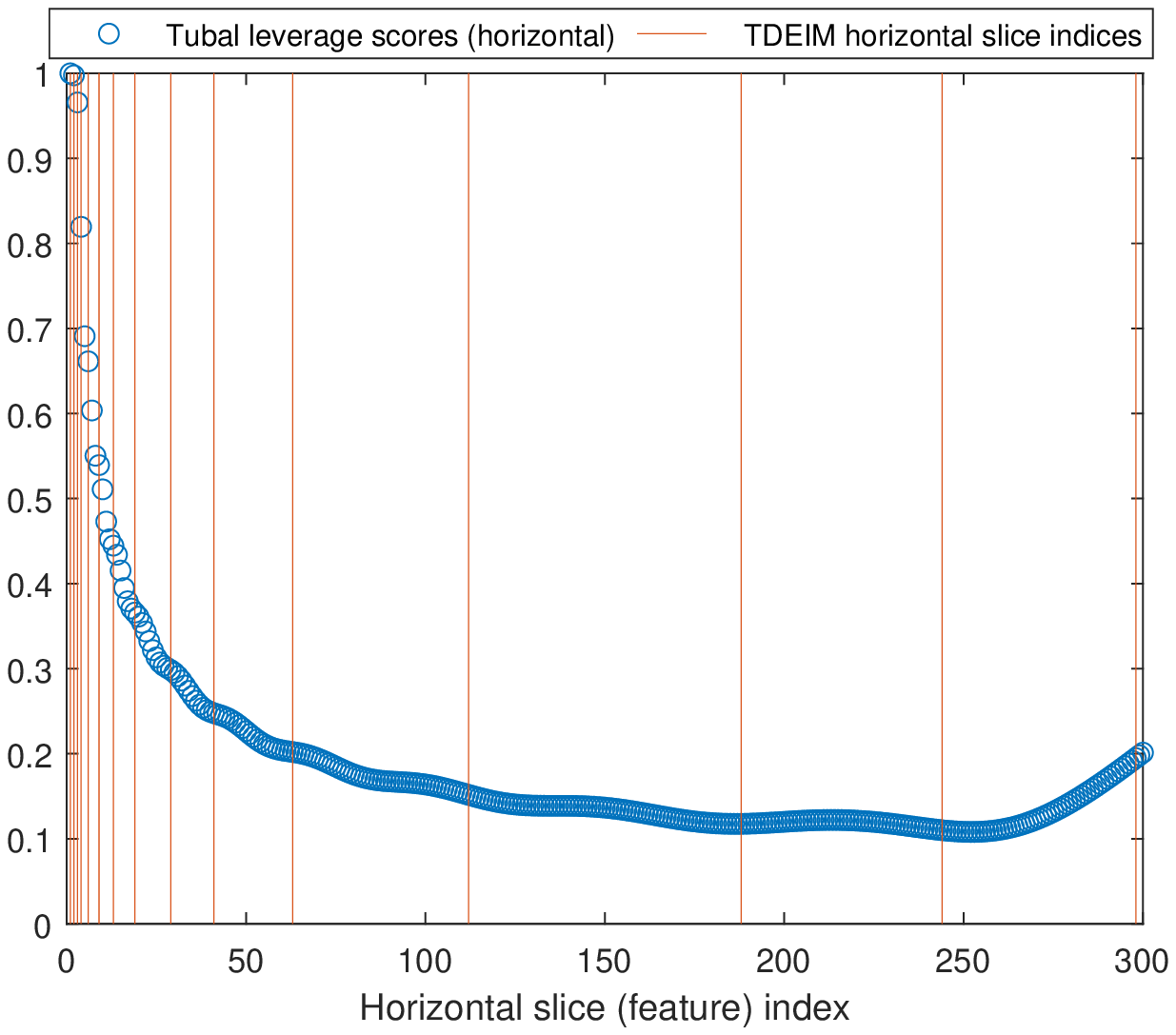}\includegraphics[width=0.5\columnwidth]{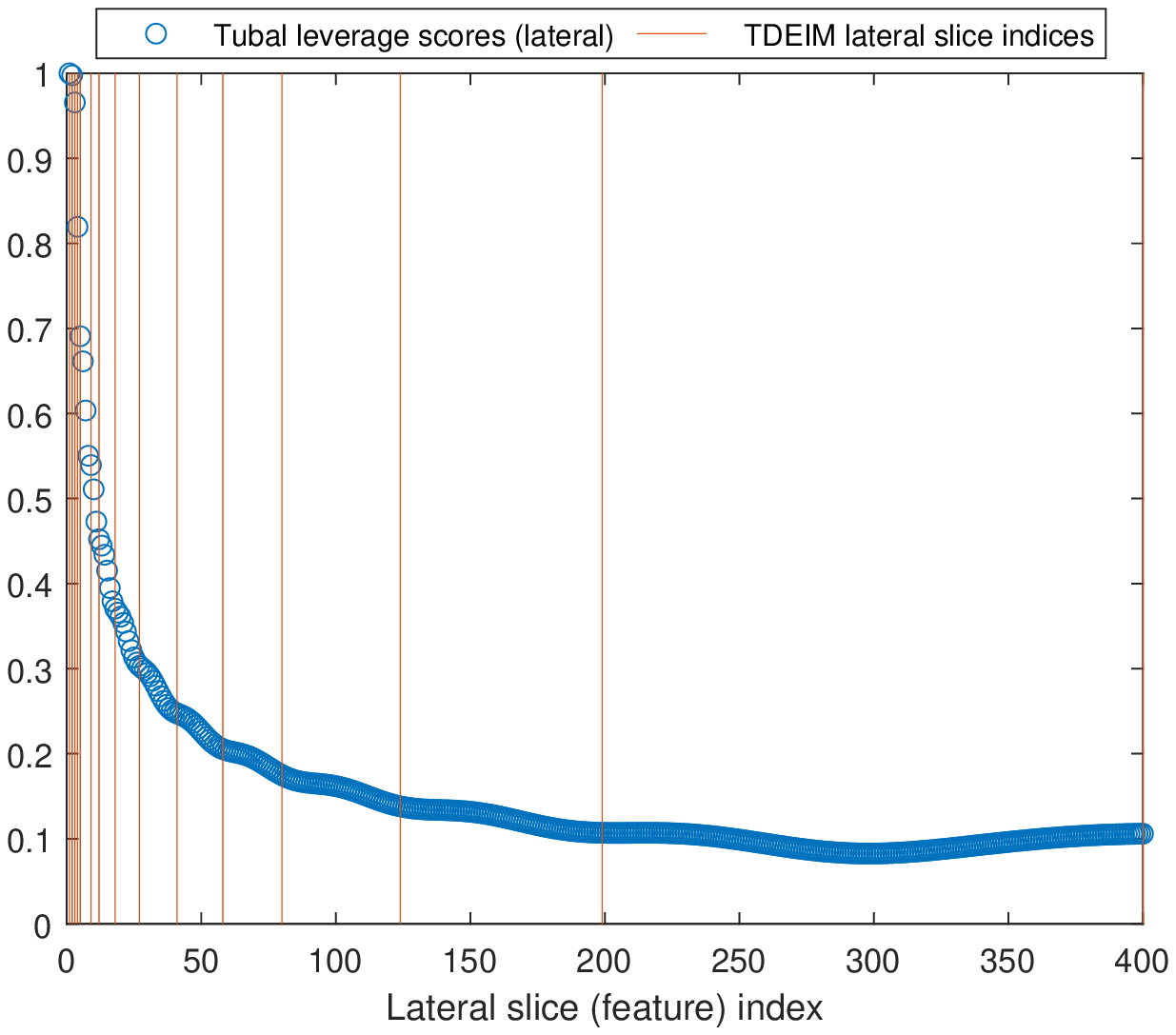}\\
\caption{\small{({\bf Left}) The selected horizontal slice indices ({\bf Right}) The selected lateral slice indices 
 for $p=3$. Top singular tensors of tubal rank $R=15$ were used for Example \ref{exa_1}.}}\label{Fig_3}
\end{center}
\end{figure}

\begin{figure}
\begin{center}
\includegraphics[width=0.5\columnwidth]{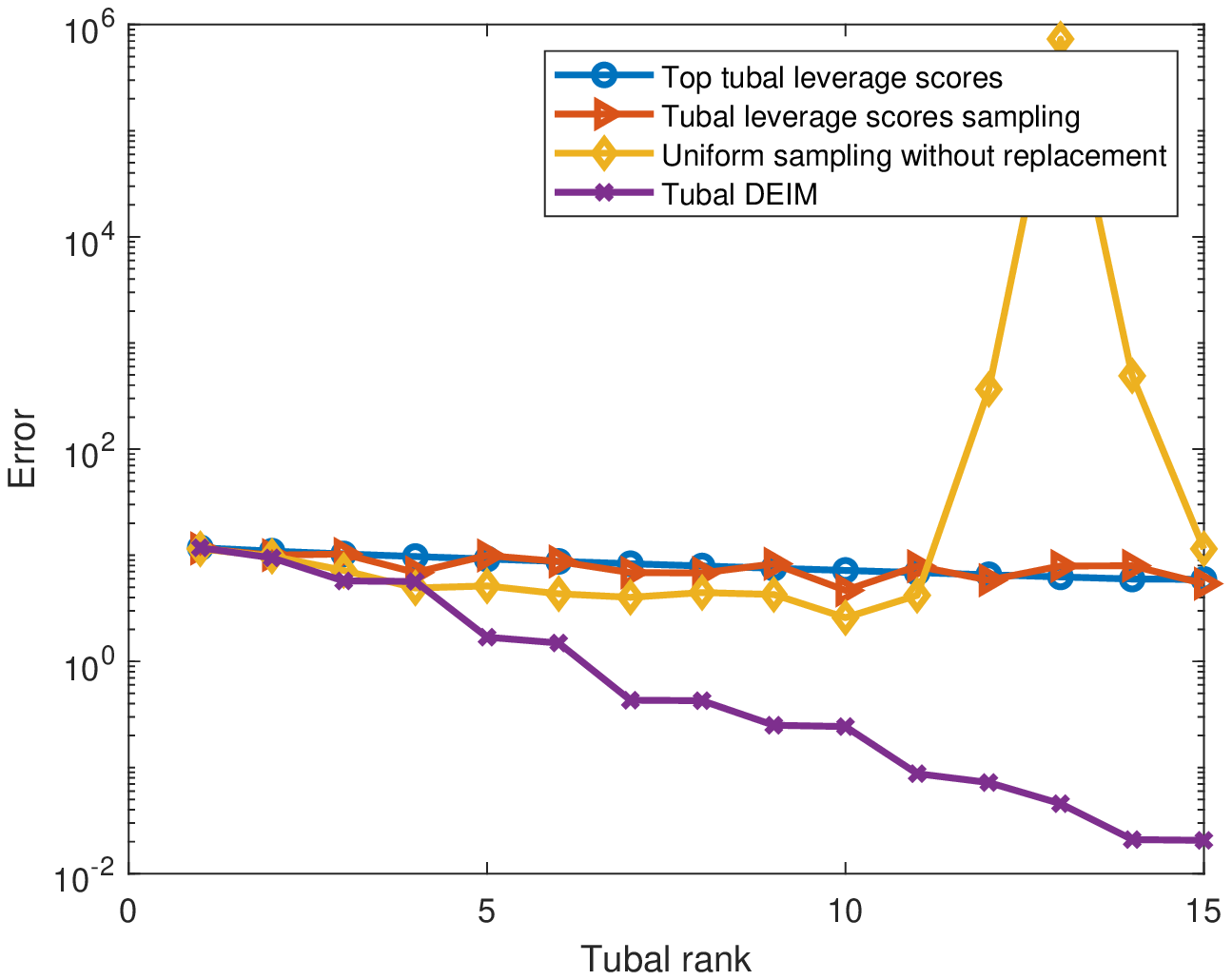}\includegraphics[width=0.5\columnwidth]{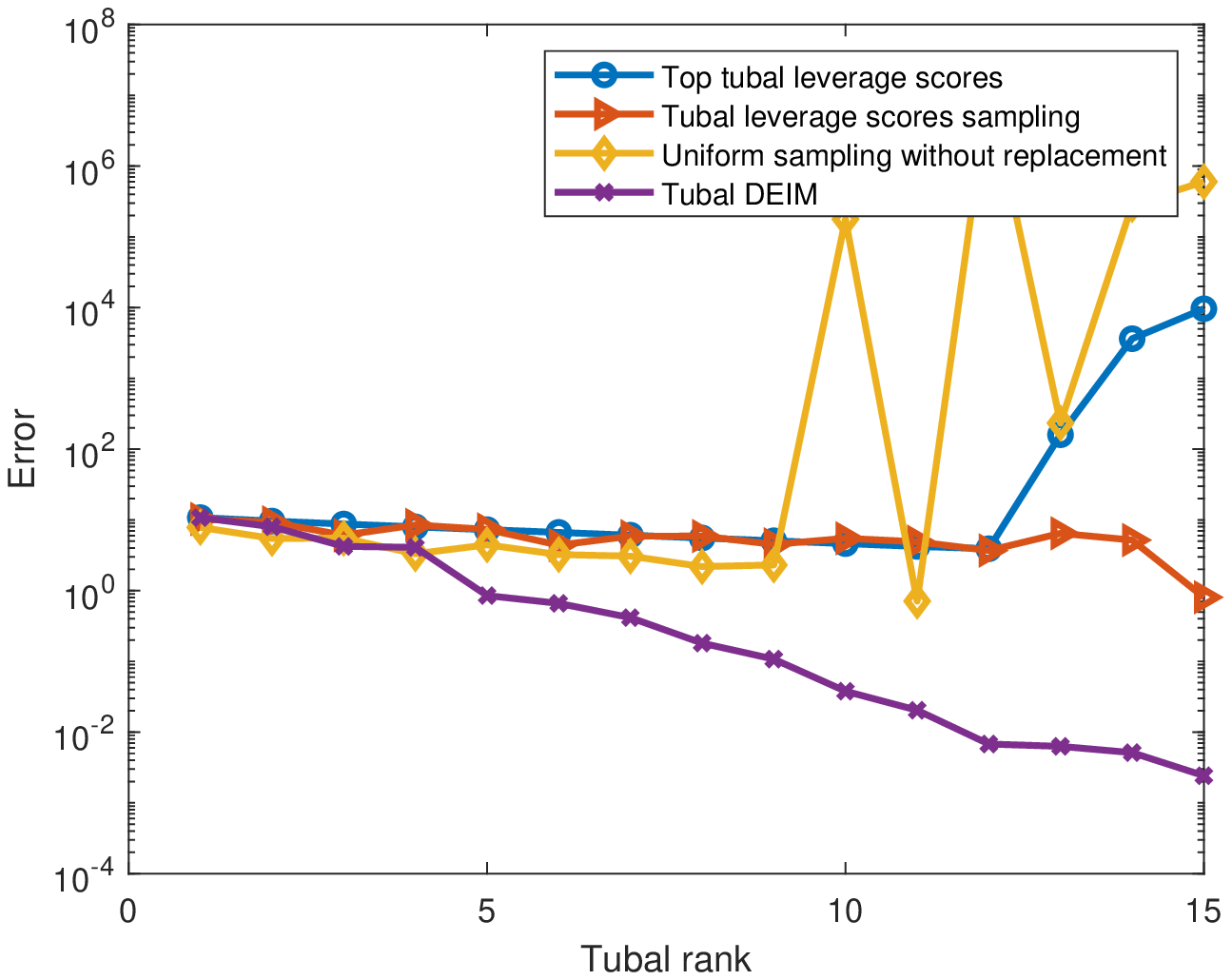}\\
\caption{\small{Errors of different sampling algorithms, ({\bf Left}) $p=5$ ({\bf Right}) $p=3$. Top singular tensors of tubal rank $R=15$ were used for Example \ref{exa_1}.} }\label{Fig_2}
\end{center}
\end{figure}

\begin{figure}
\begin{center}
\includegraphics[width=0.5\columnwidth]{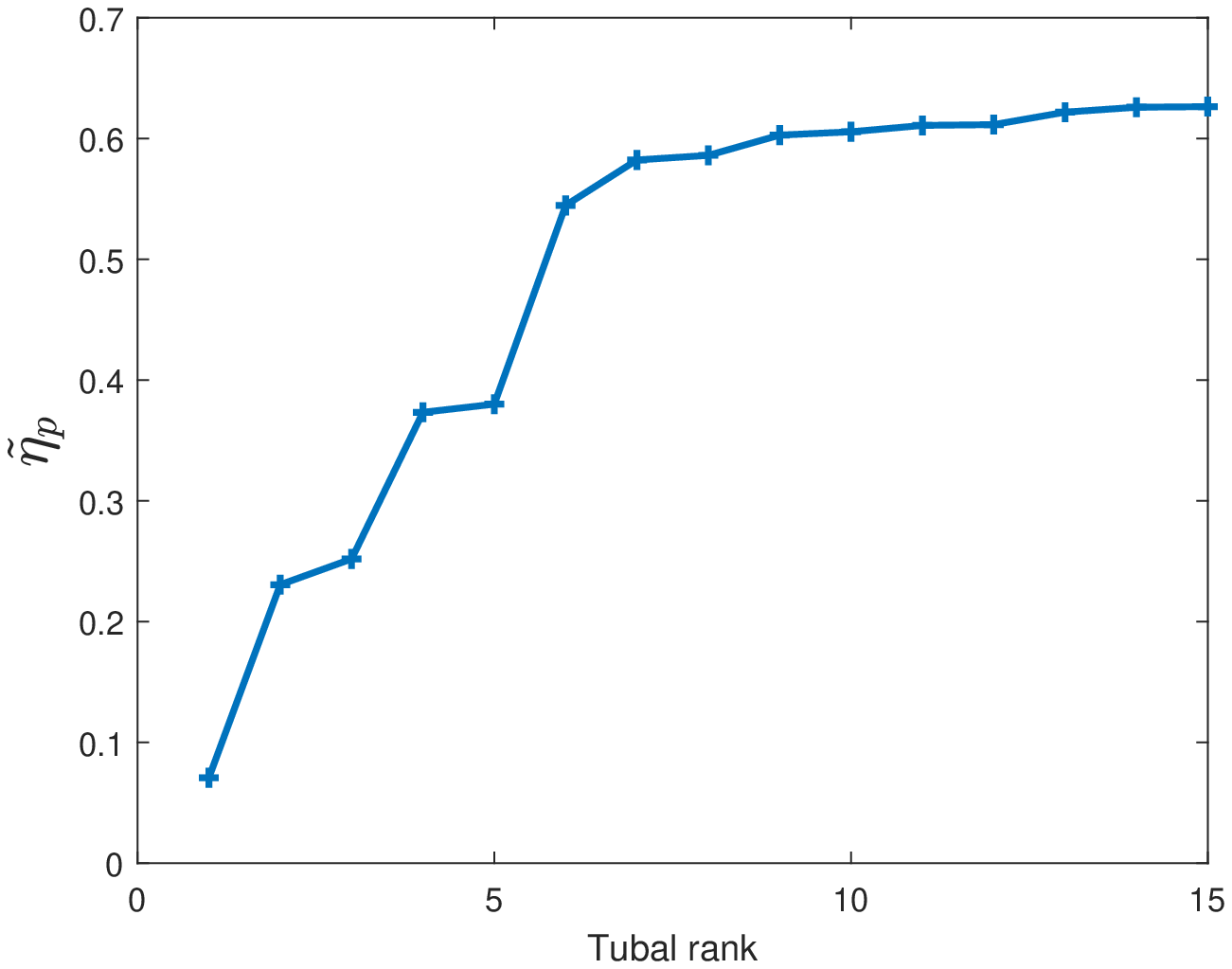}\includegraphics[width=0.5\columnwidth]{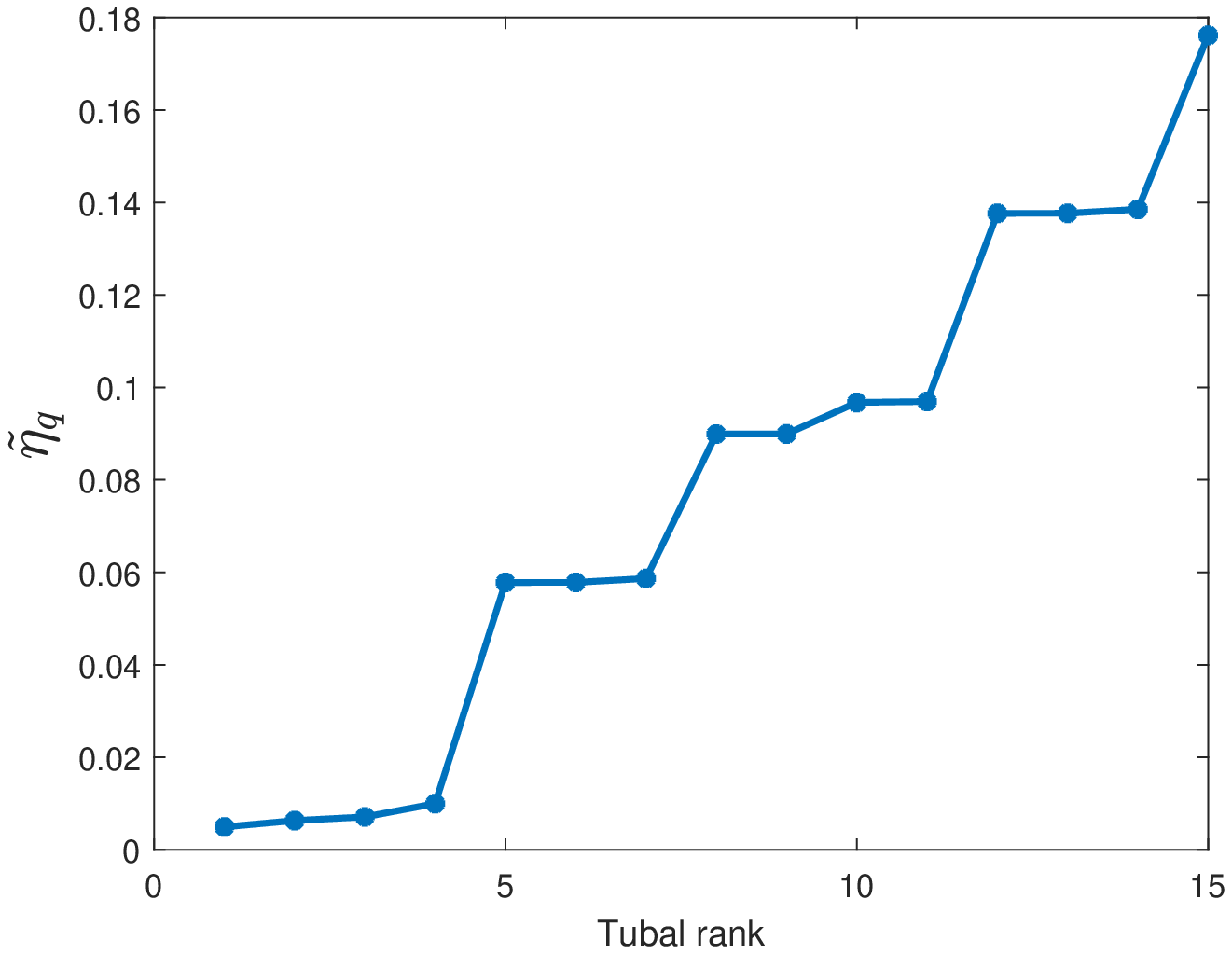}\\
\caption{\small{The error constants $\tilde{\eta}_p$ and $\tilde{\eta}_q$ for the data tensor \eqref{rand_ten} with $p=3$ for Example \ref{exa_1}.}}\label{eta}
\end{center}
\end{figure}

\begin{table}
\begin{center}
\caption{Comparing the approximation errors obtained via the TDEIM (Algorithm \ref{ALG:DEIM_ten}) and the baselines for a low tubal approximation of tubal rank $R=10$ using optimization based tensors tensors for Example \ref{exa_1}.}\label{table_func}
\vspace{0.2cm}
\begin{tabular}{|| c |c |c |c |c||} 
 \hline
  \multirow{2}{*}{Tensors} &  Top Leverage  & Top Leverage & Uniform sampling  & \multirow{2}{*}{TDEIM} \\
    &  Scores & Score sampling & without replacement &  \\
 \hline\hline
  Exponential   &  2.42e-17
 & {\bf 1.38e-17} & 2.26e-16 & 2.35e-16 \\
   \hline
  Rastrigin   &  4.20e-05
 & 1.06e-04 & 7.88e-05 & {\bf 4.13e-05} \\\hline
Booth   &  1.41e-03
 & 2.22e-04 & 1.04e-03 & {\bf 2.10e-04}\\
 \hline
 Matyas   &  3.70e-06
 & 9.47e-07 & 1.10e-06 & {\bf 4.09e-07}\\
 \hline
 Easom   &  7.67e-16
 & 4.96e-16 & 5.70e-16 & {\bf 7.16e-17}\\
 \hline
\end{tabular}
\end{center}
\end{table}

\begin{table}
\begin{center}
\caption{Comparing the running time (seconds) of the HTDEIM (Algorithm \ref{ALG:DEIM_tenlever}) and the TDEIM (Algorithm \ref{ALG:DEIM_ten}) for the data tensor \eqref{rand_ten} for Example \ref{exa_1}.}\label{Table1}
\vspace{0.2cm}
\begin{tabular}{|| c| c c c c||} 
 \hline
  & $R=2$ & $R=5$ & $R=10$ & $R=15$\\
 \hline\hline
 HTDEIM   & {\bf 12.52} & { \bf 14.354} & {\bf 17.15} & {\bf 20.90} \\ 
  \hline
  TDEIM  & 16.01 & 23.29 & 25.31 & 32.19\\
 \hline
\end{tabular}
\end{center}
\end{table}

\begin{table}
\begin{center}
\caption{Comparing the errors of the HTDEIM (Algorithm \ref{ALG:DEIM_tenlever}) and the TDEIM (Algorithm \ref{ALG:DEIM_ten}) for Example \ref{exa_1}.}\label{Table2}
\vspace{0.2cm}
\begin{tabular}{|| c| c c c c||} 
 \hline
  & $R=2$ & $R=5$ & $R=10$ & $R=15$\\
 \hline\hline
 HTDEIM  &  0.3102
 & 0.0403 & 0.0015 & 0.0014  \\ 
  \hline
  TDEIM  &  0.2902
 & 0.0308 & 0.0014 & 0.0012 \\
 \hline
\end{tabular}
\end{center}
\end{table}


\end{exa}

\begin{exa}[{\bf Image approximation}]\label{exa_2}
In this experiment, we consider the two color images ``Lena'' and ``peppers'' which are of size $256\times 256\times 3$. We used the top left and right tubal singular tensors of tubal rank $R=20$ to build the tubal leverage scores and also as a basis in the TDEIM algorithm. The tubal leverage scores and the indices, which were selected by the TDEIM are demonstrated in Figures \ref{lenaindices} and \ref{peppersindices}. The errors of the approximations obtained by the proposed and three baseline algorithms are also reported in Figure \ref{imagecompare}. Similar to the synthetic case, the better performance of the proposed algorithm is visible. In this experiment, the proposed TDEIM and top tubal leverage scores approaches were monotonically decreasing while the tubal leverage score sampling and uniform sampling methods were not. Since the tubal leverage scores were almost uniformly distributed, the TDEIM approximately samples indices for lateral and horizontal slices uniformly.     

\end{exa}

\begin{figure}
\begin{center}
\includegraphics[width=0.5\columnwidth]{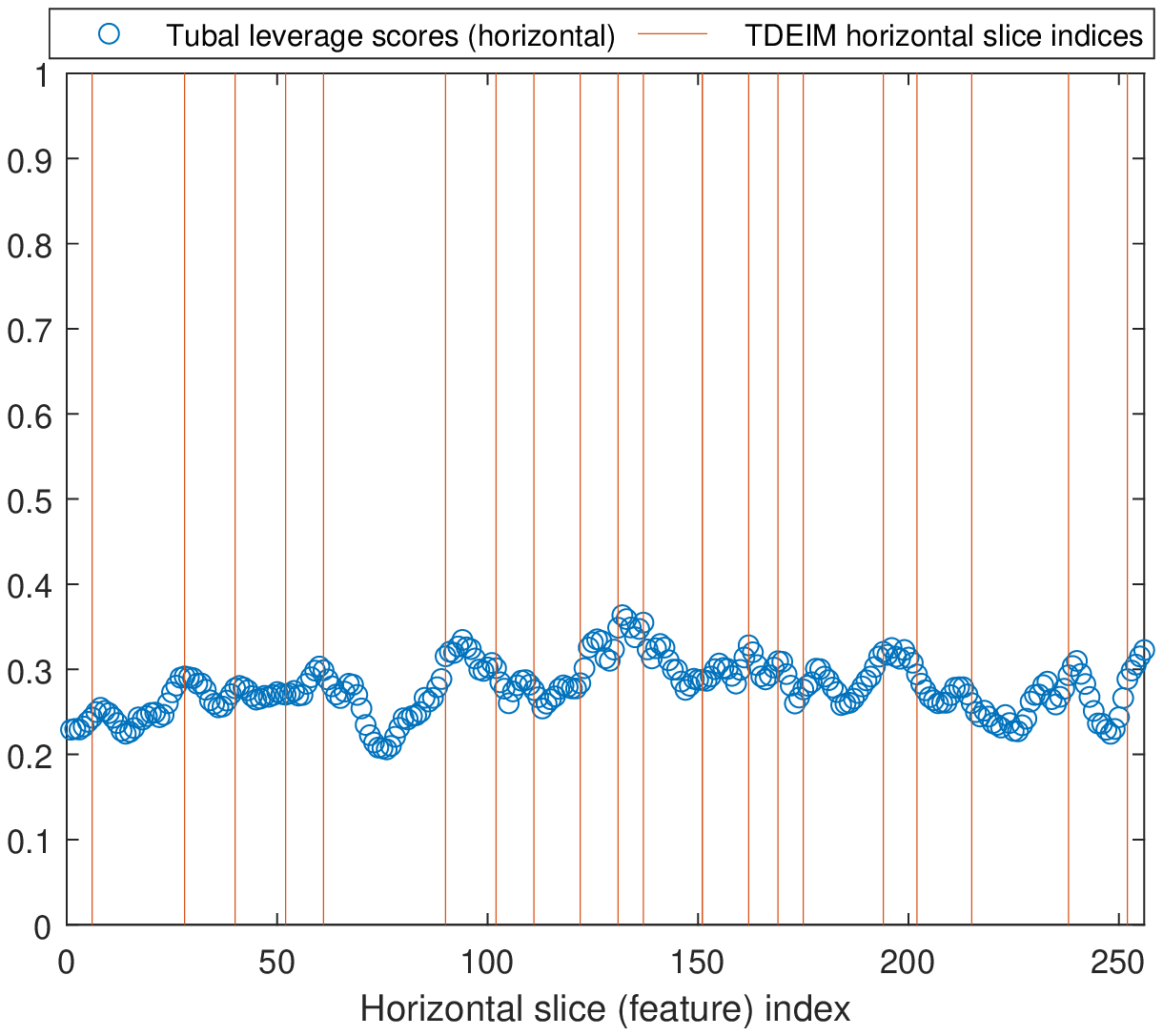}\includegraphics[width=0.5\columnwidth]{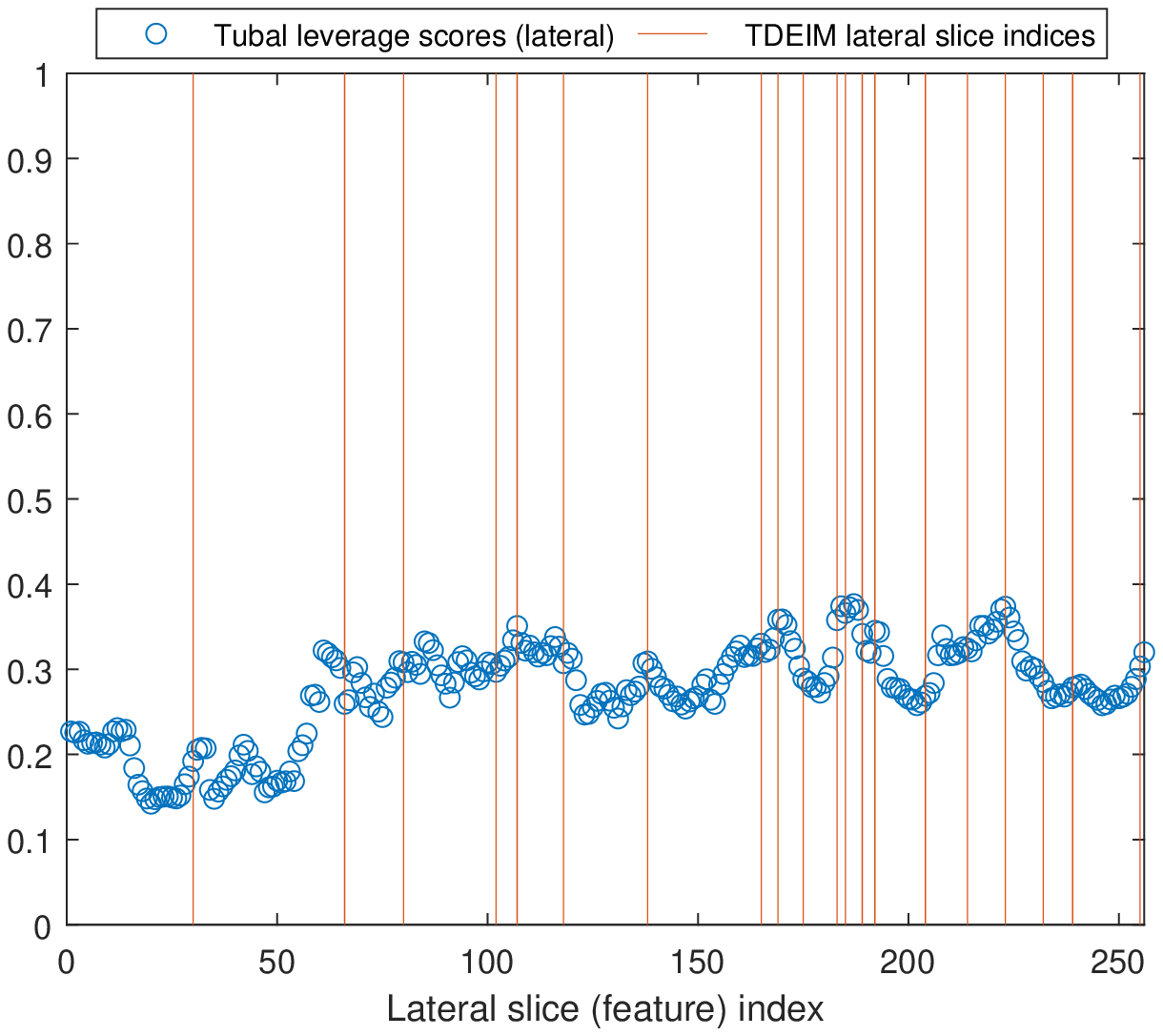}\\
\caption{\small{({\bf Left}) The selected horizontal slice indices ({\bf Right}) The selected lateral slice indices 
 for the Lena image. Top singular tensors of tubal rank $R=20$ were used for Example \ref{exa_2}.}}\label{lenaindices}
\end{center}
\end{figure}

\begin{figure}
\begin{center}
\includegraphics[width=0.5\columnwidth]{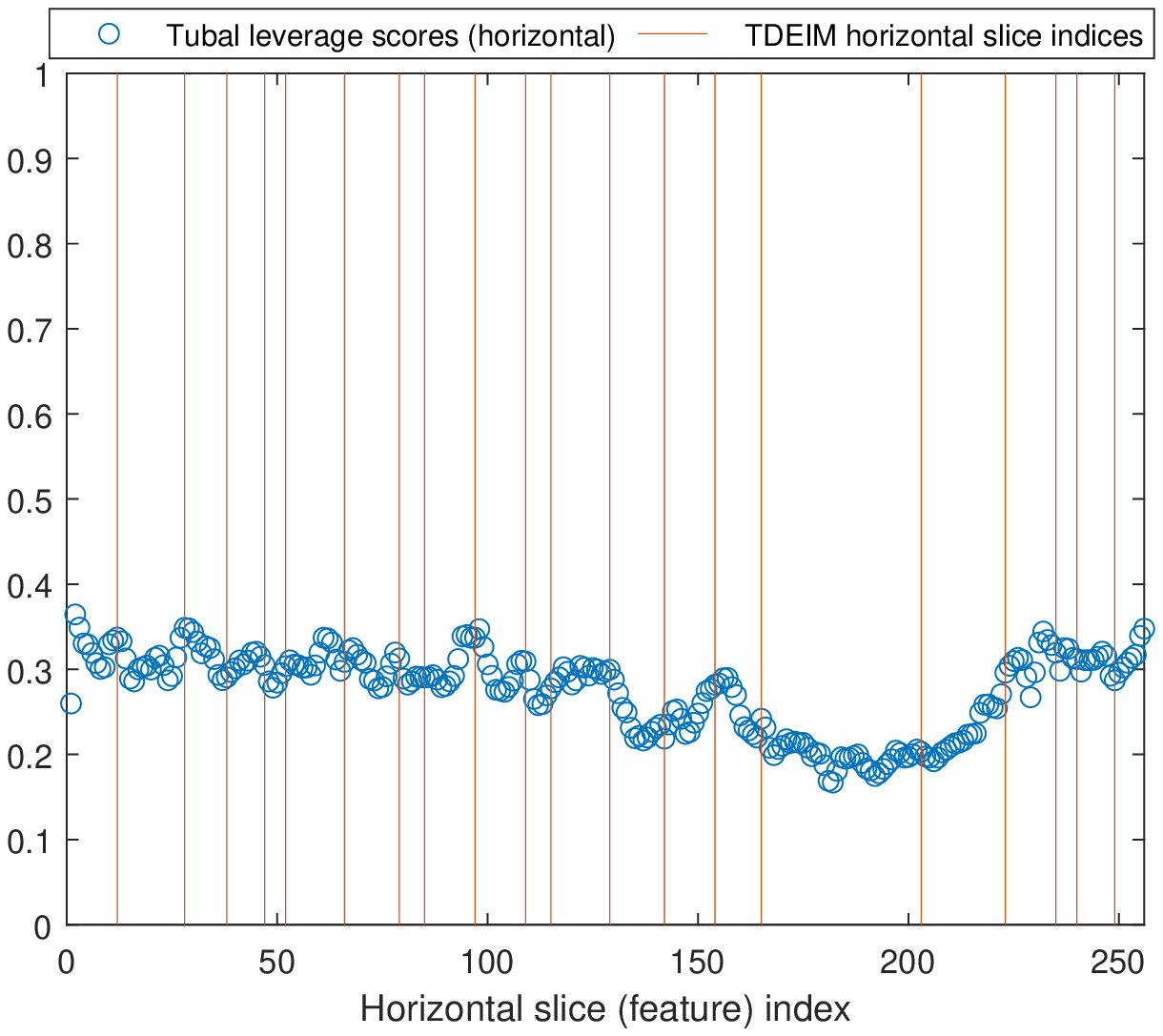}\includegraphics[width=0.5\columnwidth]{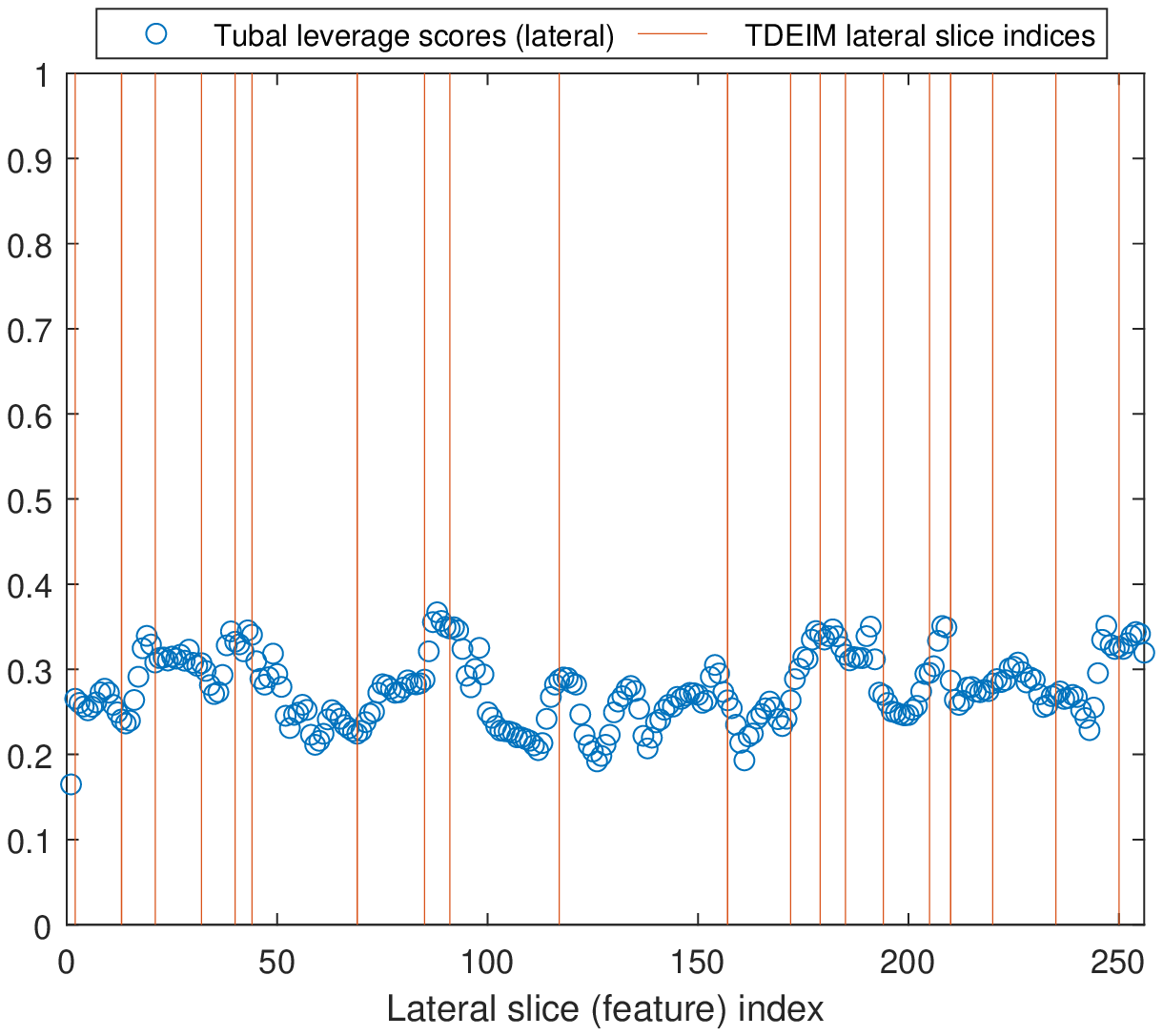}\\
\caption{\small{({\bf Left}) The selected horizontal slice indices ({\bf Right}) The selected lateral slice indices 
 for the peppers image. Top singular tensors of tubal rank $R=20$ were used for Example \ref{exa_2}.}}\label{peppersindices}
\end{center}
\end{figure}

\begin{figure}
\begin{center}
\includegraphics[width=0.5\columnwidth]{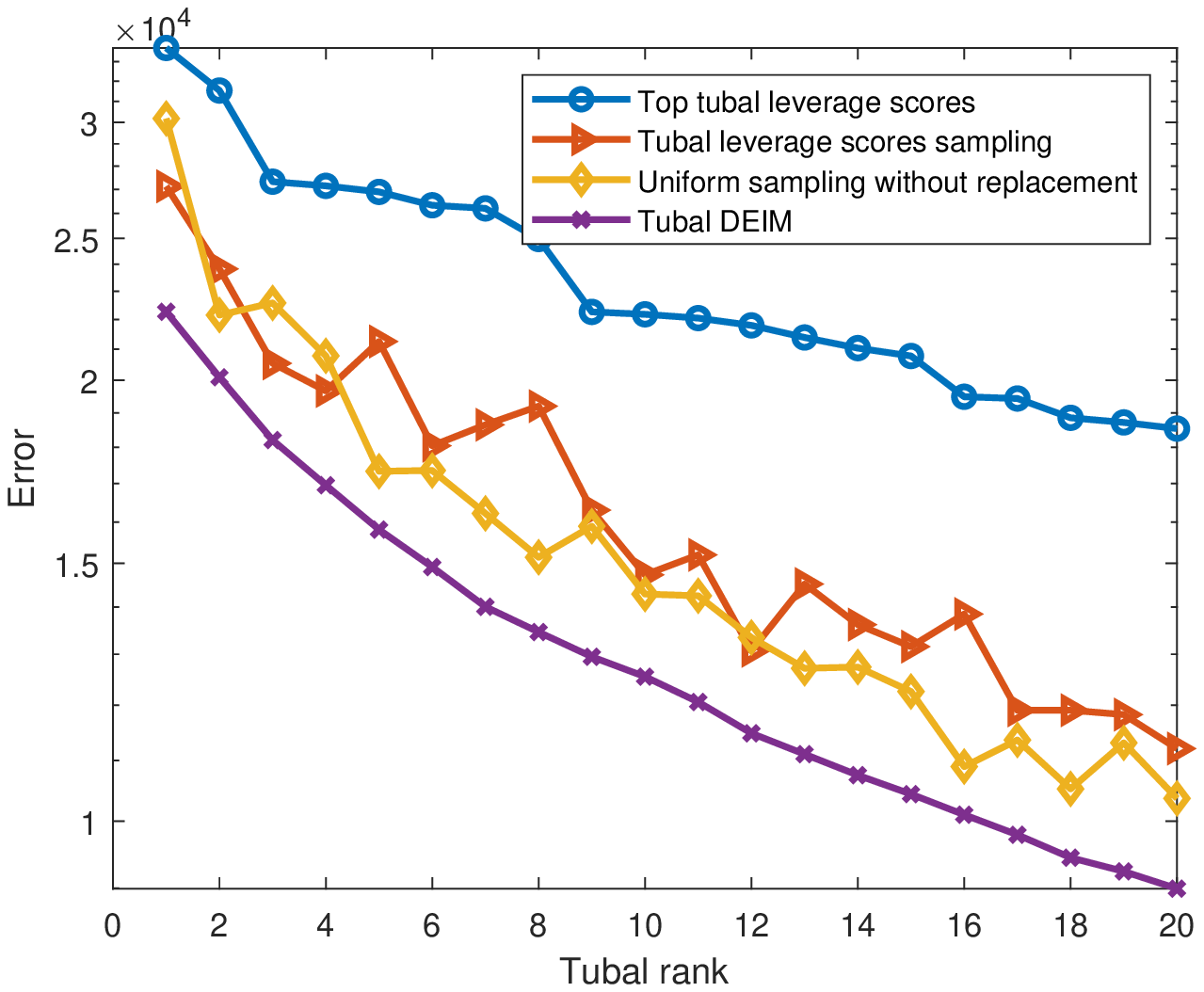}\includegraphics[width=0.5\columnwidth]{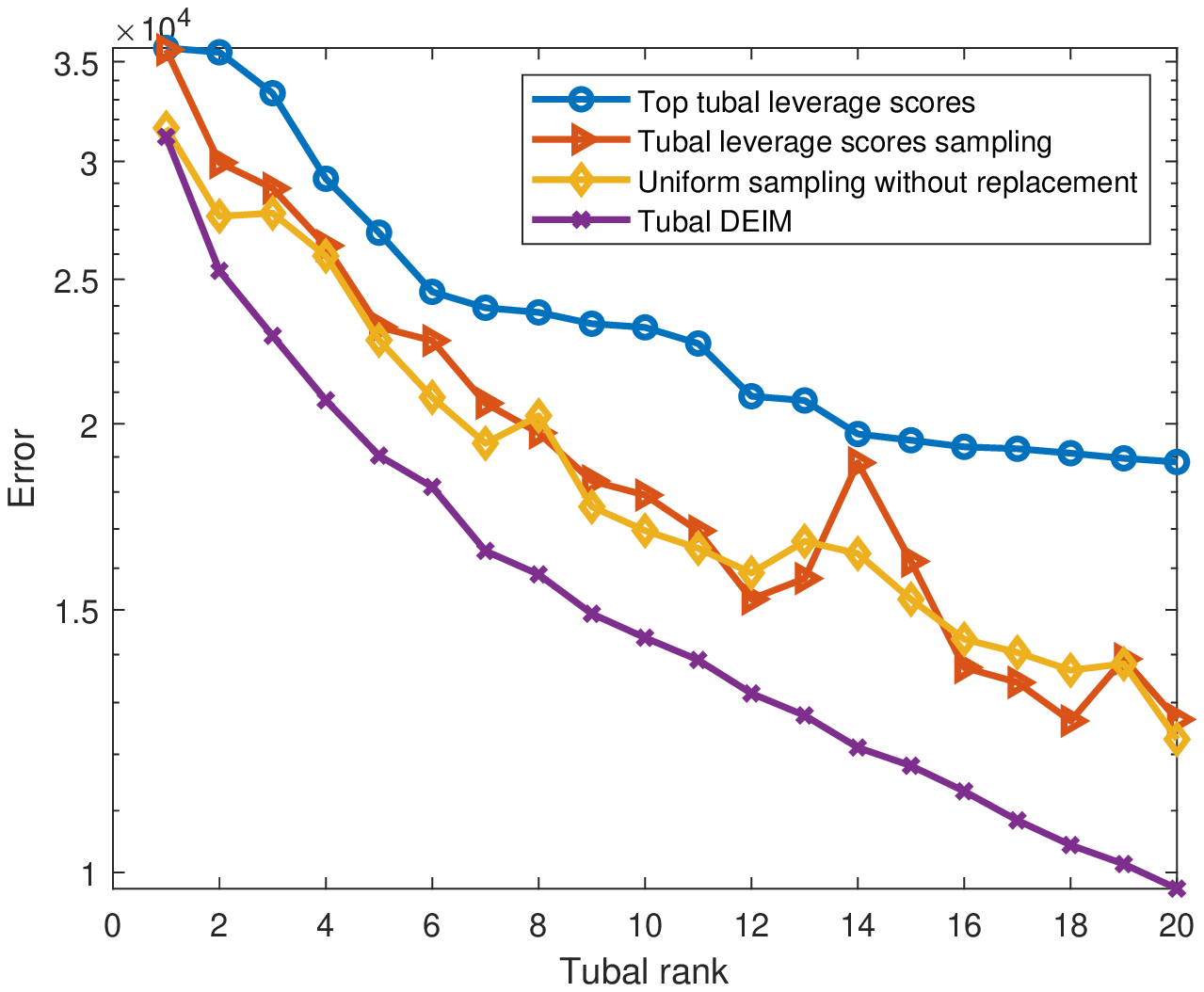}\\
\caption{\small{Errors of different sampling algorithms, ({\bf Left}) lena ({\bf Right}) peppers. Top singular tensors of tubal rank $R=20$ were used for Example \ref{exa_2}.} }\label{imagecompare}
\end{center}
\end{figure}

\begin{exa}\label{exa_3} ({\bf Video approximation})
In this experiment, we considered two videos: ``Foreman'' and ``Suzie video'' from \url{http://trace.eas.asu.edu/yuv/}. The size of Foremen video is $176\times 144\times 300$ and the size of the Suzie video is $176\times 144\times 150$. For both of them, we used tubal rank $R=40$ and computed the top singular tensors to build the tual leverage scores (horizontal and lateral) and also used them as bases in the TDEIM algorithms. The tubal leverage scores of the Foremen and Suzie video datasets and the indices, which were selected by the TDEIM are shown in Figures \ref{Foremanindices} and \ref{Suzieindices}, respectively. The accuracy of the algorithms are reported in Figure \ref{videocompare}. The results verifies the superiority of the proposed sampling method over the there widely used baseline sampling algorithms. It is interesting to note that here again since we do not have uniform distribution of tubal leverage scores, the TDEIM selects most of the indices in the region with high tubal leverage scores but all of them.
    
\end{exa}

\begin{figure}
\begin{center}
\includegraphics[width=0.5\columnwidth]{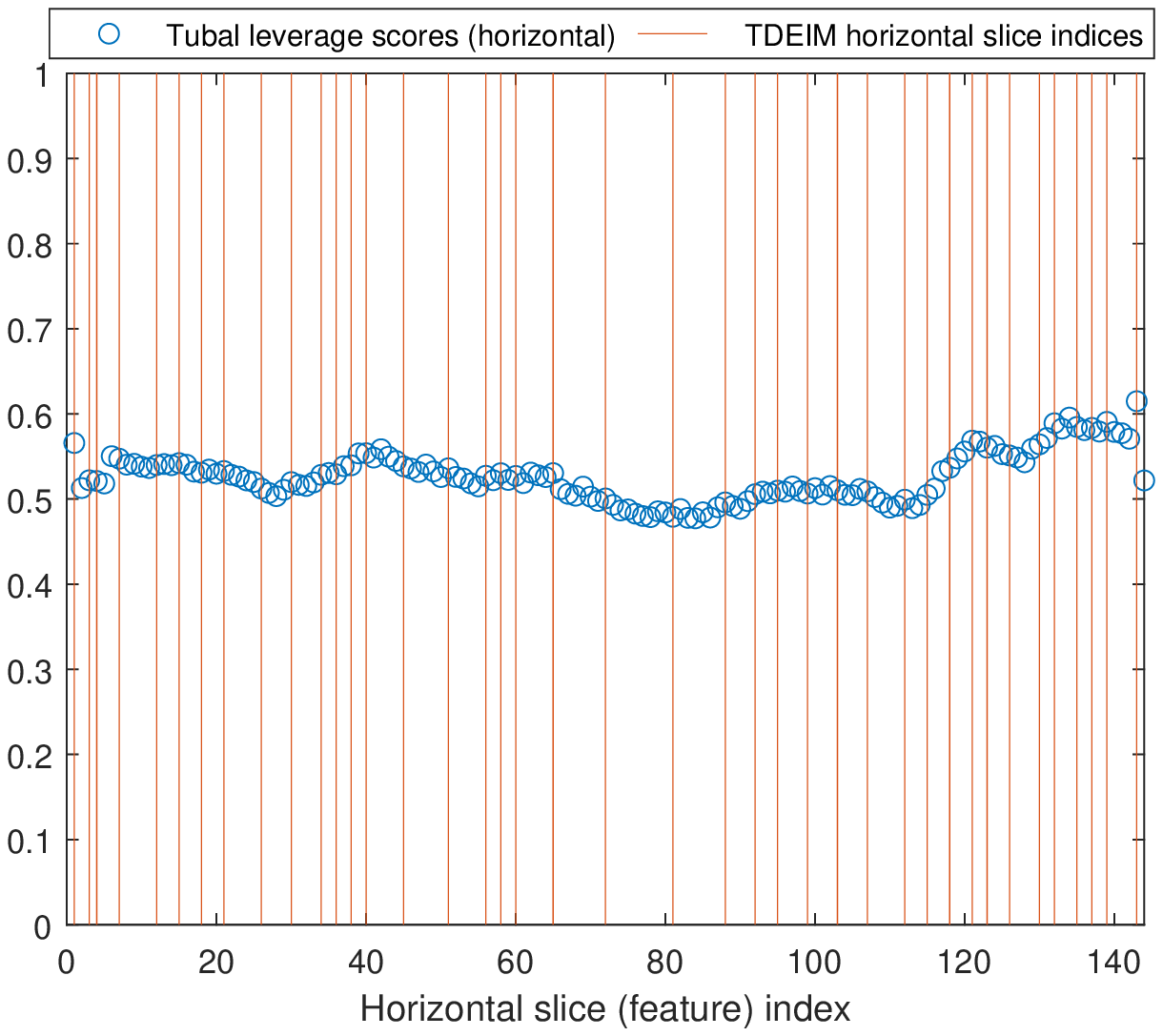}\includegraphics[width=0.5\columnwidth]{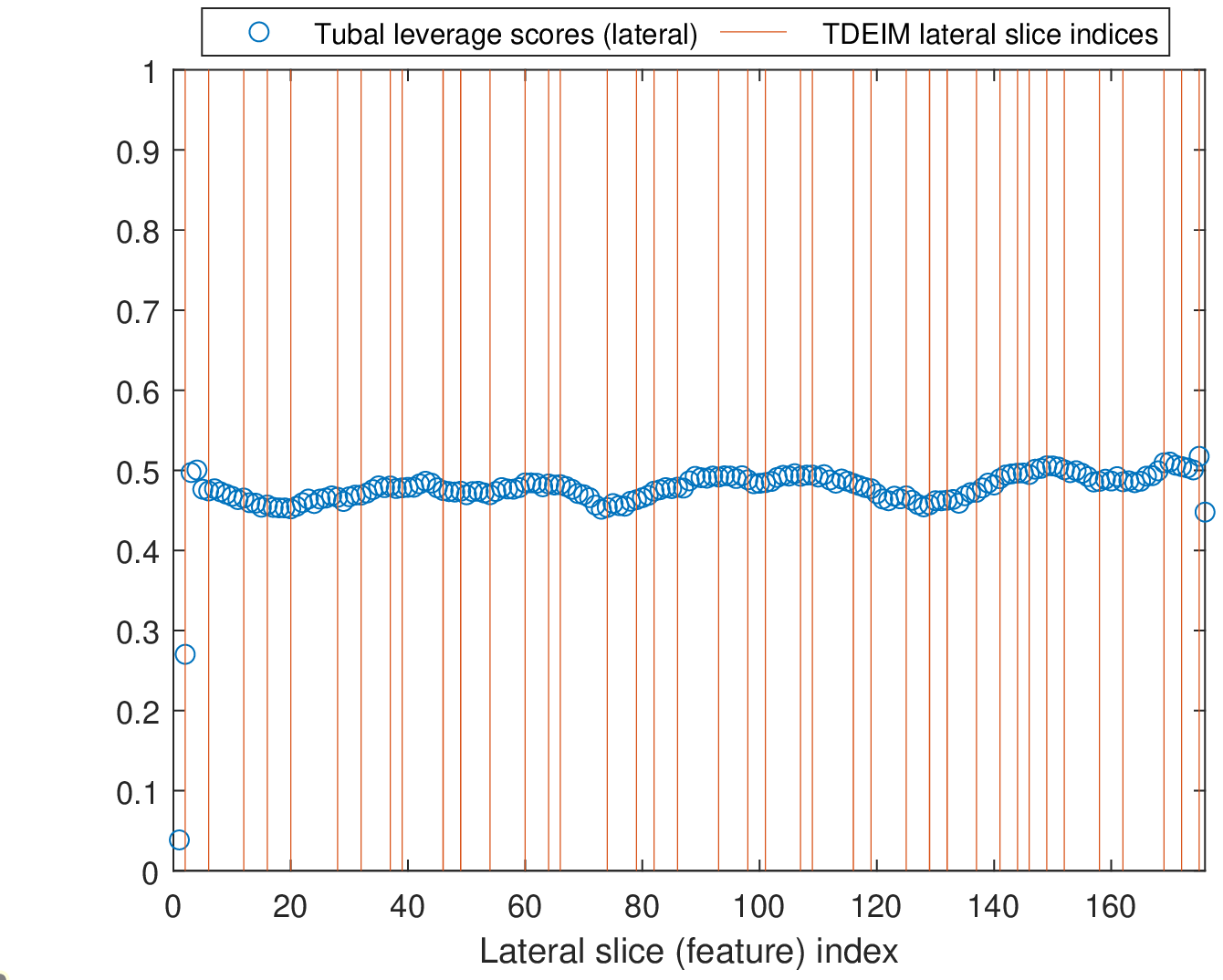}\\
\caption{\small{({\bf Left}) The selected horizontal slice indices. ({\bf Right}) The selected lateral slice indices 
 for the Foreman video. Top singular tensors of tubal rank $R=40$ were used for Example \ref{exa_3}.}}\label{Foremanindices}
\end{center}
\end{figure}

\begin{figure}
\begin{center}
\includegraphics[width=0.5\columnwidth]{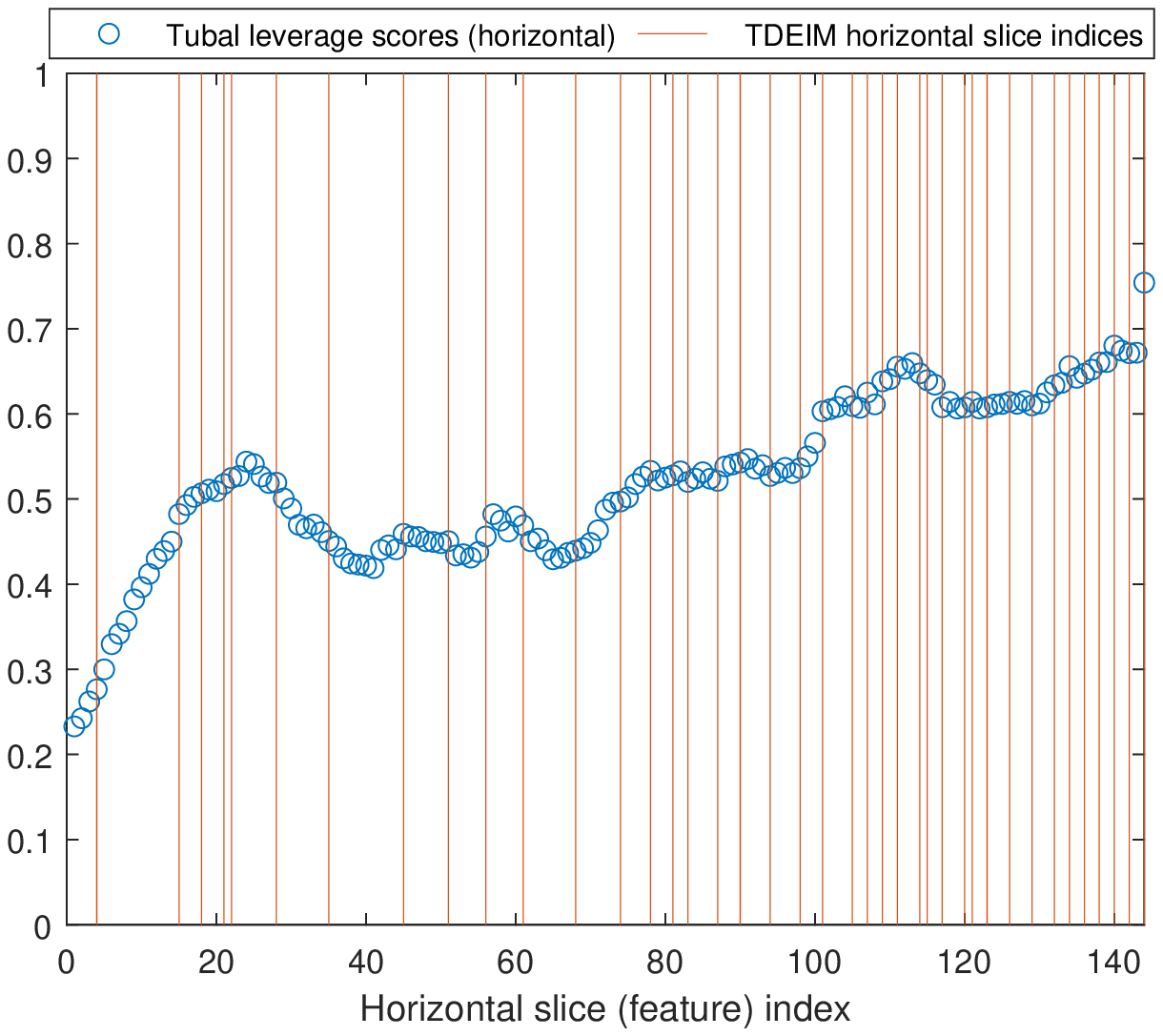}\includegraphics[width=0.5\columnwidth]{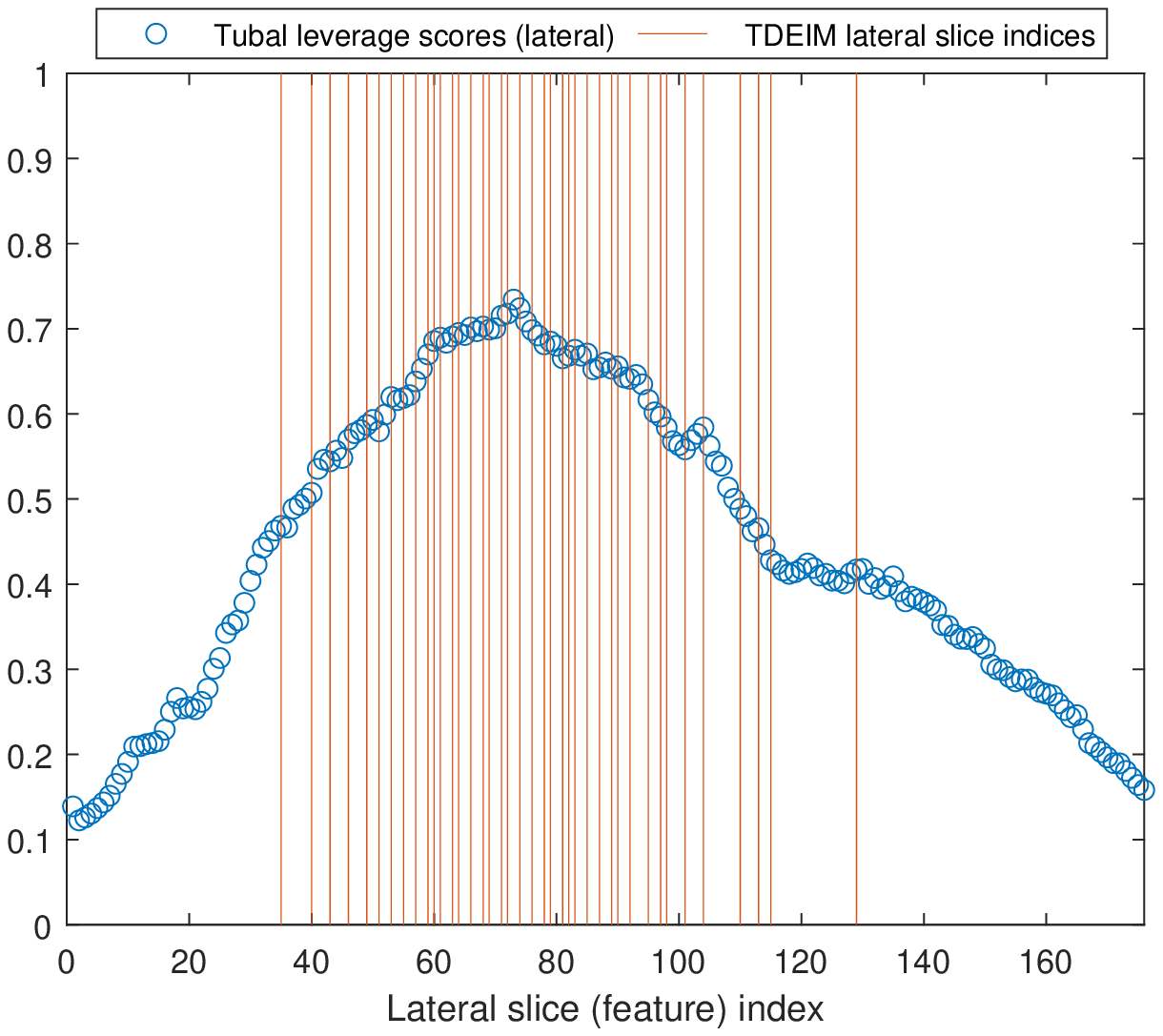}\\
\caption{\small{({\bf Left}) The selected horizontal slice indices ({\bf Right}) The selected lateral slice indices 
 for the Suzie video. Top singular tensors of tubal rank $R=40$ were used for Example \ref{exa_3}.}}\label{Suzieindices}
\end{center}
\end{figure}

\begin{figure}
\begin{center}
\includegraphics[width=0.5\columnwidth]{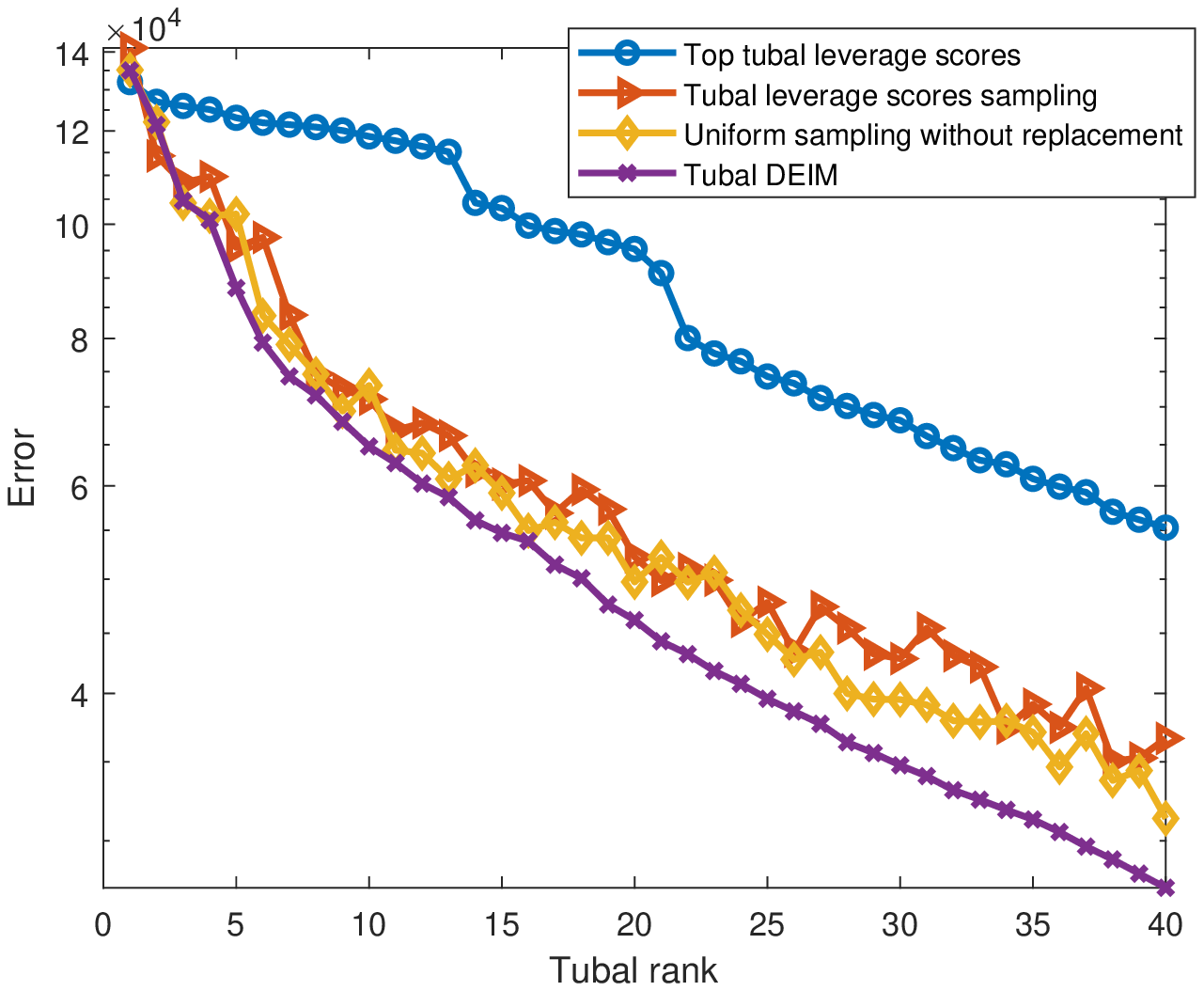}\includegraphics[width=0.5\columnwidth]{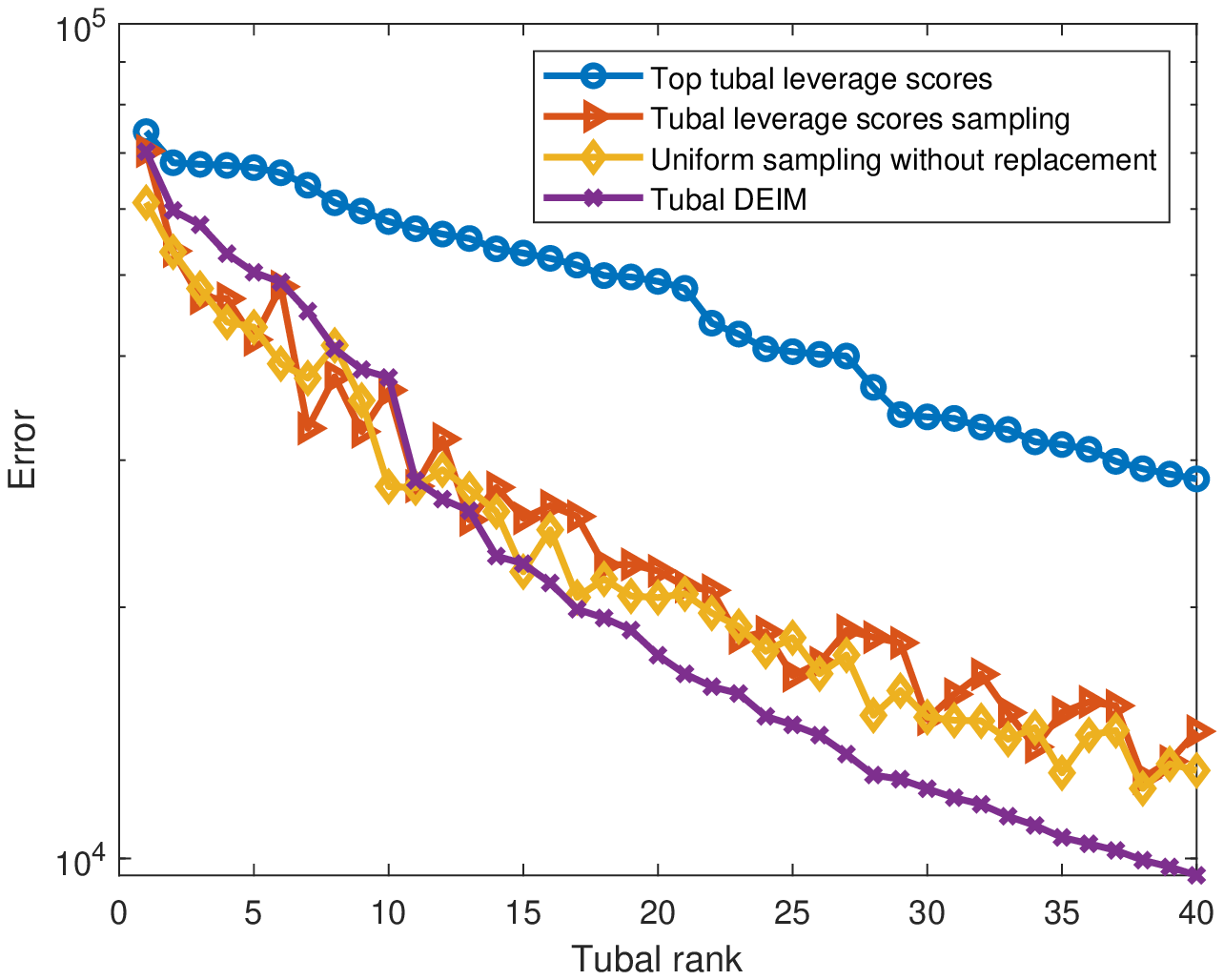}\\
\caption{\small{Errors of different sampling algorithms, ({\bf Left}) Foreman video ({\bf Right}) Suzie video. Top singular tensors of tubal rank $R=40$ were used for Example \ref{exa_3}.} }\label{videocompare}
\end{center}
\end{figure}

\begin{exa}\label{exa_5} ({\bf Classification problem})
In this experiment, we demonstrate the application of the proposed TDEIM method to the classification task on a subset of the MNIST hand-written images\footnote{\url{http://yann.lecun.com/exdb/mnist/}}, which consists of the first $1000$ images for digits $1$ and $7$. Each image of size $28 \times 28$ is transformed by Gabor wavelets \cite{lee1996image} to give an order-4 tensor of size $28 \times 28 \times
8\, ({\rm orientations}) \times 4\, ({\rm scales})$ or order-3 tensor of size $784 \times 8\,({\rm orientations}) \times
4\, ({\rm scales})$. Concatenate order-3 Gabor tensors from all images to give a new
tensor, $\cY$ of dimensions $784 \times 8 \times 4 \times N$, where $N$ is the number of images. 

{We split the data into 10 folds randomly following the 10-fold cross-validation.
For each phase of test, one fold of the data is used for test, and the rest for training.}
We train a linear discriminant analysis classifier (LDA) on the mentioned dataset \cite{hastie2009elements,fisher1936use}.
It is known that although the LDA is a simple method, it can provide decent, interpretable and robust classification results \cite{hastie2009elements,fisher1936use}. The weight tensor, which was trained by the LDA is a third order tensor of size $32\times 32\times 28$ and to build a light-weight model with lower number of parameters, we compute a low tubal rank approximation of the weight tensor by sampling some lateral and horizontal slices using the proposed TDEIM method. This leads to a faster inference time of the underlying model. 

The corresponding accuracy achieved by the low tubal rank approximation was also calculated for different numbers of sampled lateral and horizontal slices. The classification accuracy for the digits $(1,7)$ using different numbers of sampled lateral/horizontal slices are reported in Figure \ref{Accuray} (Upper). As we can see, the accuracy of 0.98 is attained for 15 horizontal and lateral slices, indicating that the classifier is fairly accurate in recognizing the digits 1 and 7. The 10-fold cross validation results for all combinations of different digits are displayed in Figure \ref{Accuray} (Bottom). The results clearly demonstrate that the lightweight model can provide correct classifications for the any combination of digits. 

\begin{figure}
\begin{center}
\includegraphics[width=0.7\columnwidth]{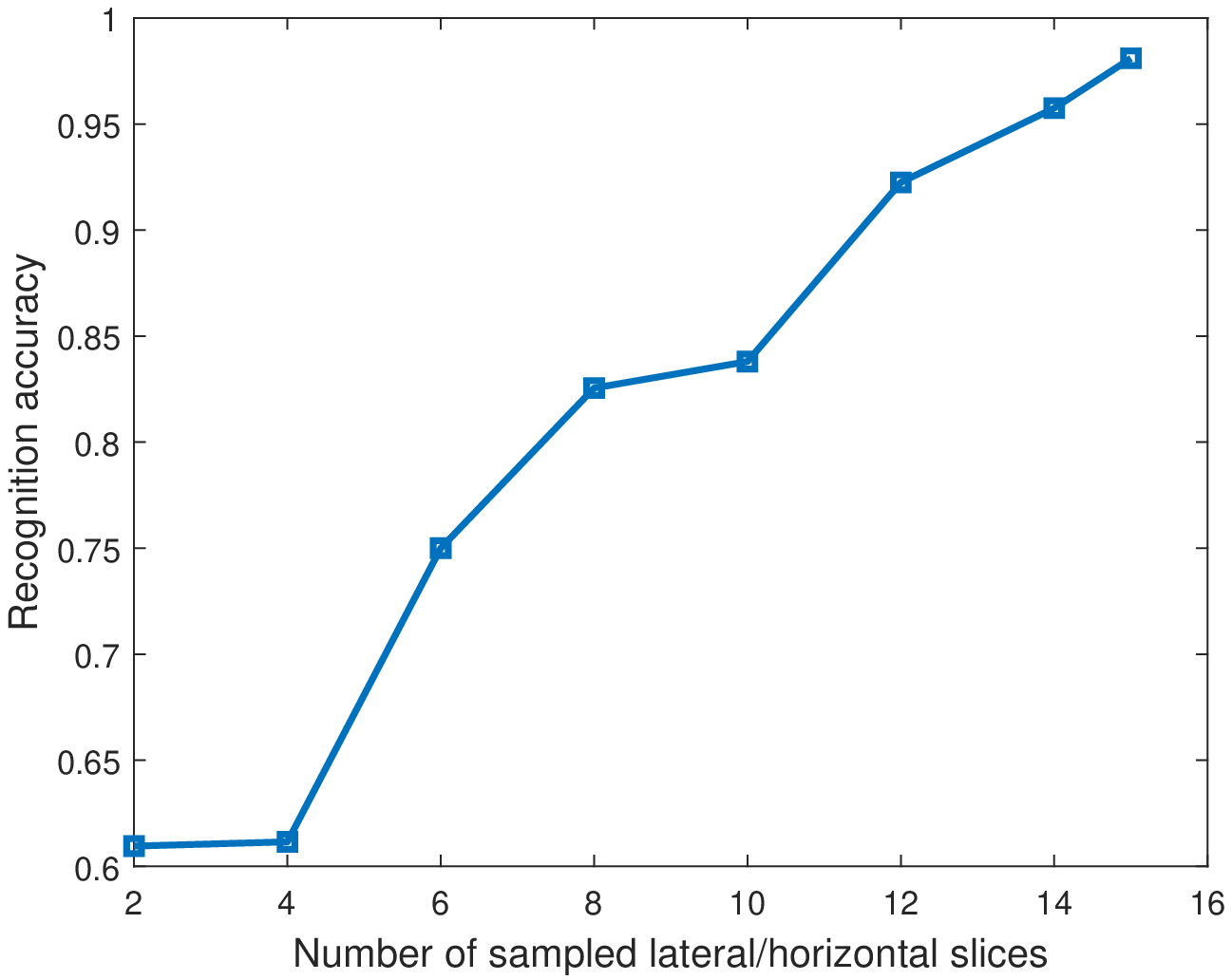}\\
\includegraphics[width=0.7\columnwidth]{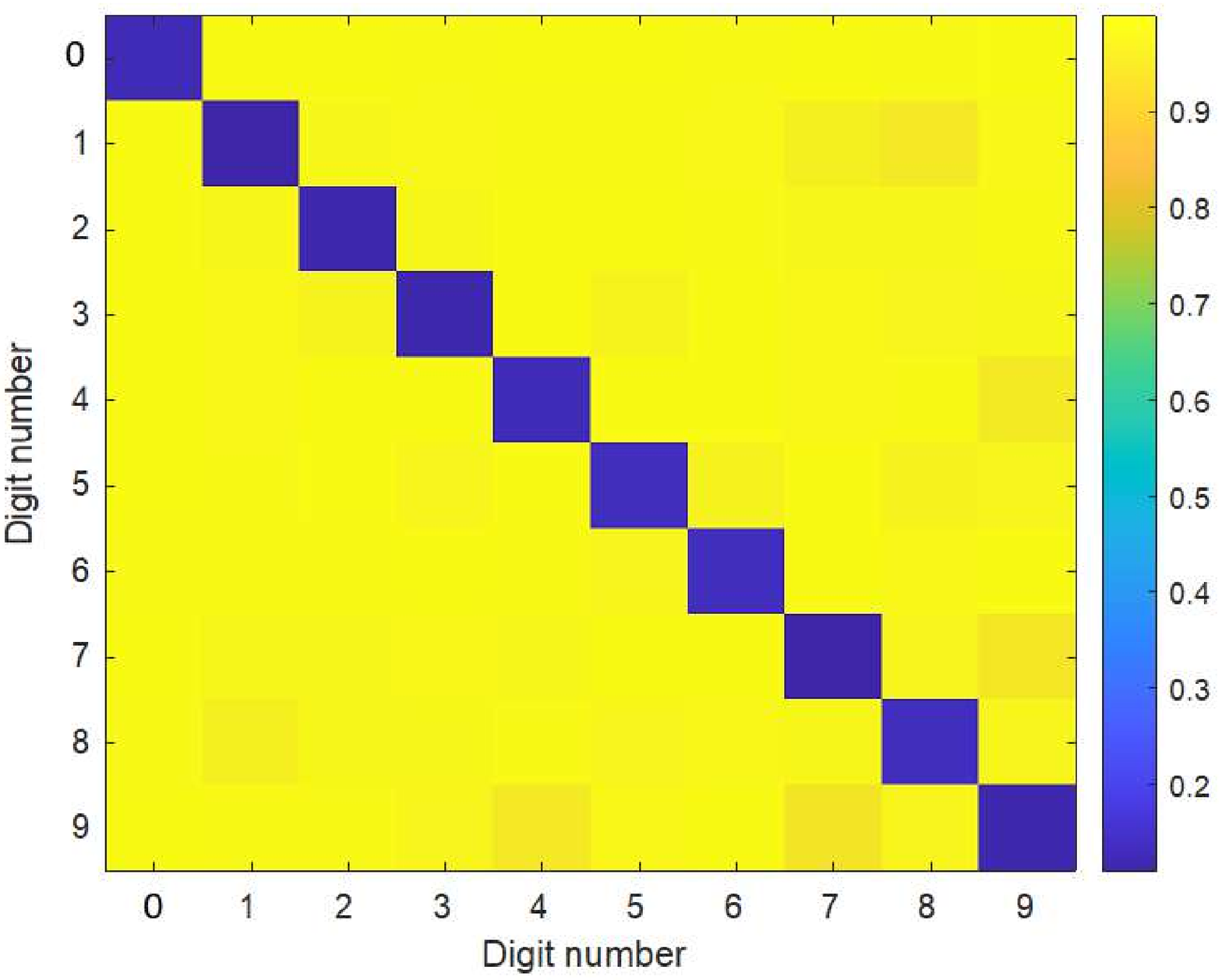}\\
\caption{\small{({\bf Upper}) The accuracy yielded by the lightweight classification model of digits $(1,7)$ using a low tubal rank approximation of the weight tensor by the proposed TDEIM approach for different numbers of sampled lateral and horizontal slices. ({\bf Bottom}) The classification accuracy of the lightweight model using a low tubal rank approximation of the weight tensor (with the tubal rank $R=15$) using the proposed TDEIM method for different combinations of digits for Example \ref{exa_5}.} }\label{Accuray}
\end{center}
\end{figure}


\end{exa}

\section{Conclusion}\label{Sec:Con}
In this paper, we extended the discrete empirical interpolation method (DEIM) to tensors based on the t-product. The tubal DEIM (TDEIM) is used to select important horizontal and lateral slices of a given third order tensor. We studied the theoretical aspects of the TDEIM and conducted simulations on synthetic and real-world datasets. The results show the better accuracy of the proposed algorithm compared to other sampling algorithms such as top tubal leverage scores, tubal leverage scores sampling and inform sampling without replacement.   

\section{Conflict of interest}
The authors declare that they have no conflict of interest.

\bibliographystyle{elsarticle-num} 
\bibliography{cas-refs}


\end{document}